\documentclass[12pt]{article}
\usepackage{fullpage}
\usepackage{amsmath,amssymb,amsfonts,amsthm}
\usepackage[all,arc,knot,poly]{xy}
\usepackage{enumerate}
\usepackage{mathrsfs}
\usepackage{mathtools}
\usepackage{amsxtra}
\usepackage{needspace}
\usepackage{graphicx}
\usepackage{setspace}
\usepackage{authblk}

\DeclareMathOperator{\Tr}{\operatorname{Tr}}

\theoremstyle{definition}
\newtheorem{theorem}{Theorem}[section]
\newtheorem{definition}[theorem]{Definition}
\newtheorem{lemma}[theorem]{Lemma}
\newtheorem{proposition}[theorem]{Proposition}

\newtheorem{corollary}[theorem]{Corollary}

\title{A duality map for quantum cluster varieties \protect\\from surfaces}

\author{Dylan G.L. Allegretti}
\affil{Department of Mathematics, Yale University, 10 Hillhouse Ave, New Haven, CT 06511, U.S.A.}

\author{Hyun Kyu Kim}
\affil{School of Mathematics, Korea Institute for Advanced Study, 85 Hoegiro, Dongdaemun-gu, Seoul~02455, Republic of Korea}

\date{}

\begin{document}

\maketitle

\begin{abstract}
We define a canonical map from a certain space of laminations on a punctured surface into the quantized algebra of functions on a cluster variety. We show that this map satisfies a number of special properties conjectured by Fock and Goncharov. Our construction is based on the ``quantum trace'' map introduced by Bonahon and~Wong.
\end{abstract}

\tableofcontents

\section{Introduction}

\subsection{Fock and Goncharov's duality map}

In their seminal paper~\cite{IHES}, Fock and Goncharov associated, to a punctured surface $S$, a pair of moduli spaces denoted $\mathcal{A}_{SL_2,S}$ and~$\mathcal{X}_{PGL_2,S}$. These spaces are closely related to moduli spaces of $SL_2$- and $PGL_2$-local systems, respectively. Fock and Goncharov showed that certain ideas from Teichm\"uller theory can be understood in terms of these objects~\cite{IHES}.

One of the main results of~\cite{IHES} (see also~\cite{dual}) was the existence of a duality between the spaces $\mathcal{A}_{SL_2,S}$ and $\mathcal{X}_{PGL_2,S}$. More precisely, Fock and Goncharov defined a space $\mathcal{A}_{SL_2,S}(\mathbb{Z}^t)$ which is a tropicalization of $\mathcal{A}_{SL_2,S}$, and they constructed a canonical map 
\[
\mathbb{I}:\mathcal{A}_{SL_2,S}(\mathbb{Z}^t)\rightarrow\mathbb{Q}(\mathcal{X}_{PGL_2,S})
\]
from this space into the algebra of functions on $\mathcal{X}_{PGL_2,S}$. The spaces $\mathcal{A}_{SL_2,S}(\mathbb{Z}^t)$ and $\mathcal{X}_{PGL_2,S}$ admit natural coordinates $\{a_i\}_{i\in I}$ and $\{X_i\}_{i\in I}$, parametrized by the set $I$ of edges of an ideal triangulation of~$S$. If we number these edges so that $I=\{1,\dots,n\}$, then we have the following result.

\begin{theorem}[\cite{IHES}, Theorem~12.2]
\label{thm:introclassicalproperties}
The canonical functions $\mathbb{I}(\ell)$ defined by the above construction satisfy the following properties:
\begin{enumerate}
\item For any choice of ideal triangulation, $\mathbb{I}(\ell)$ is a Laurent polynomial in the coordinates $X_i$ with highest term $X_1^{a_1}\dots X_n^{a_n}$ where $a_i$ is the coordinate of $\ell$ associated to the edge~$i$.

\item The coefficients of the Laurent polynomial $\mathbb{I}(\ell)$ are positive integers.

\item For any $\ell$,~$\ell'\in\mathcal{A}_{SL_2,S}(\mathbb{Z}^t)$, we have
\[
\mathbb{I}(\ell)\mathbb{I}(\ell')=\sum_{\ell''\in\mathcal{A}_{SL_2,S}(\mathbb{Z}^t)}c(\ell,\ell';\ell'')\mathbb{I}(\ell'')
\]
where $c(\ell,\ell';\ell'')$ are nonnegative integers and only finitely many terms are nonzero.
\end{enumerate}
\end{theorem}

One of the important features of $\mathcal{X}_{PGL_2,S}$ is that this space is equipped with a natural Poisson structure and can be canonically quantized~\cite{ensembles,dilog}. To do this, we replace the commutative functions $X_i^{\pm1}$ by a collection of noncommuting variables, also denoted $X_i^{\pm1}$, which satisfy the relations 
\[
X_iX_j=q^{2\varepsilon_{ij}}X_jX_i
\]
where $q$ is a complex number and $\varepsilon_{ij}$ is a skew-symmetric matrix determined by the ideal triangulation. Thus we associate to each ideal triangulation $T$ of~$S$ a noncommutative algebra $\mathcal{X}_T^q$ sometimes called the Chekhov-Fock algebra in the literature. If we let $\widehat{\mathcal{X}}_T^q$ denote the noncommutative fraction field of $\mathcal{X}_T^q$, then for triangulations $T$ and $T'$, there is a map $\Phi_{TT'}^q:\widehat{\mathcal{X}}_{T'}^q\rightarrow \widehat{\mathcal{X}}_T^q$ of skew fields which reduces to the classical transition function when~$q=1$ and satisfies the relation $\Phi_{TT''}^q=\Phi_{TT'}^q\circ\Phi_{T'T''}^q$ for triangulations $T$, $T'$, and~$T''$. These noncommutative algebras and the maps between them form what Fock and Goncharov call a \emph{quantum cluster variety}~\cite{dilog}.

\subsection{The present work}

The purpose of the present paper is to construct a quantum version of the map $\mathbb{I}$. More precisely, we prove the following theorem:

\begin{theorem}
\label{thm:intromain}
For any ideal triangulation $T$ of $S$, there exists a map  $\widehat{\mathbb{I}}^q=\widehat{\mathbb{I}}_T^q:\mathcal{A}_{SL_2,S}(\mathbb{Z}^t)\rightarrow\mathcal{X}_T^q$ satisfying the following properties:
\begin{enumerate}
\item The Laurent polynomial $\mathbb{I}(\ell)$ agrees with $\widehat{\mathbb{I}}^q(\ell)$ when~$q=1$.

\item The highest term of $\widehat{\mathbb{I}}^q(\ell)$ is 
\[
q^{-\sum_{i<j}\varepsilon_{ij}a_ia_j}X_1^{a_1}\dots X_n^{a_n}
\]
where $a_i$ is the coordinate of~$\ell$ associated to the edge~$i$.

\item Each $\widehat{\mathbb{I}}^q(\ell)$ is a Laurent polynomial in the variables $X_i$ with coefficients in $\mathbb{Z}[q,q^{-1}]$.

\item Let $*$ be the canonical involutive antiautomorphism of $\mathcal{X}_T^q$ that fixes each $X_i$ and sends~$q$ to~$q^{-1}$. Then $*\widehat{\mathbb{I}}^q(\ell)=\widehat{\mathbb{I}}^q(\ell)$.

\item For any $\ell$,~$\ell'\in\mathcal{A}_{SL_2,S}(\mathbb{Z}^t)$, we have 
\[
\widehat{\mathbb{I}}^q(\ell)\widehat{\mathbb{I}}^q(\ell')=\sum_{\ell''\in\mathcal{A}_{SL_2,S}(\mathbb{Z}^t)}c^q(\ell,\ell';\ell'')\widehat{\mathbb{I}}^q(\ell'')
\]
where $c^q(\ell,\ell';\ell'')\in\mathbb{Z}[q,q^{-1}]$ and only finitely many terms are nonzero.

\item Let $q$ be a primitive $N$th root of unity for odd $N$. Then we have the identity 
\[
\widehat{\mathbb{I}}^q(N\cdot\ell)(X_1,\dots,X_n)=\widehat{\mathbb{I}}^{\, 1}(\ell)(X_1^N,\dots,X_n^N).
\]

\item Let $\ell$ be a point of the tropical space $\mathcal{A}_{SL_2,S}(\mathbb{Z}^t)$ with coordinates $b_1,\dots,b_n$, and let $a_1,\dots,a_n$ be integers satisfying $\sum_j\varepsilon_{ij}a_j=0$ for all $i\in I$. If $\ell'\in\mathcal{A}_{SL_2,S}(\mathbb{Z}^t)$ has coordinates $a_i+b_i$ for~$i\in I$, then 
\[
\widehat{\mathbb{I}}^q(\ell') = q^{-\sum_{i<j}\varepsilon_{ij}a_ia_j}\prod_iX_i^{a_i}\cdot\widehat{\mathbb{I}}^q(\ell).
\]
\end{enumerate}
For triangulations $T$ and $T'$, the maps $\widehat{\mathbb{I}}_T^q$ and $\widehat{\mathbb{I}}_{T'}^q$ are related by $\widehat{\mathbb{I}}_T^q=\Phi_{TT'}^q\circ\widehat{\mathbb{I}}_{T'}^q$.
\end{theorem}

The properties listed in this theorem were conjectured by Fock and Goncharov in Conjecture~12.4 of~\cite{IHES} and Conjecture~4.8 of~\cite{ensembles}. The quantum duality map $\widehat{\mathbb{I}}^q$ is expected to be important in physics, where it is related to the ``protected spin character'' introduced in~\cite{GMN}. Our construction may also lead to interesting extensions of the results of~\cite{GHKK} on canonical bases for cluster algebras.

In~\cite{IHES} and~\cite{ensembles}, Fock and Goncharov conjectured various positivity properties of the map~$\widehat{\mathbb{I}}^q$. In particular, they conjectured that each $c^q(\ell,\ell',\ell'')$ is a Laurent polynomial in~$q$ with \emph{positive} integral coefficients. While we do not have a proof of this statement, we believe it is closely related to various conjectures and results from the literature, particularly the product-to-sum formula of~\cite{productsum} and Conjecture~4.20 of~\cite{Thurston}.

The basic idea of the proof of Theorem~\ref{thm:intromain} is to apply a construction of Bonahon and Wong from~\cite{BonahonWong}. This construction involves a noncommutative algebra $\mathcal{S}^A(S)$ called the skein algebra, which depends on a quantization parameter~$A$ and delivers the algebra of regular functions on the $SL_2(\mathbb{C})$-character variety of $S$ in the classical limit $A=-1$. Bonahon and Wong construct a natural map 
\[
\Tr_T^\omega:\mathcal{S}^A(S)\rightarrow\mathcal{X}_T^\omega
\]
from this algebra to the Chekhov-Fock algebra where $A=\omega^{-2}$.

The skein algebra $\mathcal{S}^A(S)$ is defined as an algebra generated by isotopy classes of framed links in the 3-manifold $S\times[0,1]$ modulo certain relations. This object bears some resemblance to the space $\mathcal{A}_{SL_2,S}(\mathbb{Z}^t)$ that we considered above. Indeed, Fock and Goncharov showed in~\cite{IHES} that the set $\mathcal{A}_{SL_2,S}(\mathbb{Z}^t)$ parametrizes laminations on the surface $S$. A lamination is defined in~\cite{IHES} as a collection of finitely many simple, disjoint, closed curves on $S$ with numerical weights, subject to certain conditions and equivalence relations.

If we are given a closed curve $\ell$ on the surface $S$, we can lift it to a framed link in $S\times[0,1]$ and apply the map $\Tr_T^\omega$ to get an element of the Chekhov-Fock algebra. We extend this idea to get a map $\mathcal{A}_{SL_2,S}(\mathbb{Z}^t)\rightarrow\mathcal{X}_T^q$ which satisfies the properties in Theorem~\ref{thm:intromain}.

\subsection{Organization}

The rest of this paper is organized as follows. In Section~\ref{sec:Preliminaries}, we review the results we need from~\cite{BonahonWong,IHES} and establish some technical lemmas for later use. We begin by recalling the definitions of $\mathcal{A}_{SL_2,S}(\mathbb{Z}^t)$ and $\mathcal{X}_{PGL_2,S}$ and the classical duality theorem from~\cite{IHES}. We then define the Chekhov-Fock and skein algebras and review the construction of Bonahon and Wong from~\cite{BonahonWong}.

In Section~\ref{sec:TheQuantumCanonicalMap}, we define our canonical map using $\Tr_T^\omega$ and show that it has the correct classical limit. We study the highest term of this map and its invariance under the involution~$*$. We then derive a version of the product expansion formula from Theorem~\ref{thm:intromain}.

Finally, in Section~\ref{sec:MainResult}, we prove our main results. The most technical step in this section involves switching between the Chekhov-Fock algebras $\mathcal{X}_T^\omega$ and $\mathcal{X}_T^q$ where $q=\omega^4$. For this, we introduce a notion of ``parity'' and apply this concept to elements of the algebra $\mathcal{X}_T^\omega$.

\section{Preliminaries}
\label{sec:Preliminaries}

\subsection{Laminations and moduli spaces}

We begin by introducing some terminology related to surfaces. Our treatment here is based on~\cite{IHES}.

\begin{definition}
A \emph{decorated surface} is a compact oriented surface with boundary together with a finite (possibly empty) collection of marked points on its boundary.
\end{definition}

Given a decorated surface $S$, we can shrink those boundary components without marked points to get a surface $S'$ with punctures and boundary where every boundary component contains at least one marked point.

\begin{definition}
Let $S$ be a decorated surface. An \emph{ideal triangulation} $T$ of~$S$ is a triangulation of the surface $S'$ described in the preceding paragraph whose vertices are the marked points and the punctures. We will write $I$ for the set of all edges of the ideal triangulation~$T$.
\end{definition}

The following definition describes a matrix which encodes the combinatorics of an ideal triangulation~$T$.

\begin{definition}
For $i$,~$j\in I$, let $a_{ij}$ denote the number of angular sectors delimited by $i$ and~$j$ in the triangles of $T$ with $i$ coming first counterclockwise, and put
\[
\varepsilon_{ij}\coloneqq a_{ij}-a_{ji}.
\]
\end{definition}

For the rest of this section, $S$ will denote a fixed decorated surface with no marked points. A curve on $S$ will be called \emph{peripheral} if it is retractable to a boundary component.

\begin{definition}[\cite{IHES}, Definition~12.1]
A \emph{rational bounded lamination} on~$S$ is the homotopy class of a collection of finitely many simple nonintersecting noncontractible closed curves on~$S$ with rational weights, subject to the following conditions and equivalence relations:
\begin{enumerate}
\item The weight of a curve is nonnegative unless the curve is peripheral.
\item A lamination containing a curve of weight zero is equivalent to the lamination with this curve removed.
\item A lamination containing homotopic curves of weights $a$ and $b$ is equivalent to the lamination with one curve removed and the weight $a+b$ on the other.
\end{enumerate}
The set of all rational bounded laminations on $S$ is denoted $\mathcal{A}_L(S,\mathbb{Q})$.
\end{definition}

From now on, we will refer to rational bounded laminations simply as \emph{laminations}. We will denote by $\mathcal{A}_L(S,\mathbb{Z})$ the set of all laminations $\ell\in\mathcal{A}_L(S,\mathbb{Q})$ such that $\ell$ can be represented by a collection of curves with integral weights.

There is a natural set of coordinate functions on $\mathcal{A}_L(S,\mathbb{Q})$ defined as follows. Fix a lamination $\ell\in\mathcal{A}_L(S,\mathbb{Q})$ and an ideal triangulation of~$S$. Deform the curves of $\ell$ so that each curve intersects each edge of the triangulation in the minimal number of points. Then we can define 
\[
a_i(\ell)\coloneqq\frac{1}{2}\mu_i(\ell)
\]
where $\mu_i(\ell)$ is the total weight of curves that intersect the edge $i$.

\begin{proposition}[\cite{IHES}, Theorem~12.1]
The numbers $a_i$ ($i\in I$) provide a bijection 
\[
\mathcal{A}_L(S,\mathbb{Q})\rightarrow\mathbb{Q}^{|I|}.
\]
\end{proposition}

Note that this proposition gives $\mathcal{A}_L(S,\mathbb{Q})$ the natural structure of a $\mathbb{Z}$-module. If $\ell$ is any element of~$\mathcal{A}_L(S,\mathbb{Z})$, we can write 
\[
\ell=\sum_ik_i\ell_i
\]
where $\{\ell_i\}_i$ is a collection of curves representing $\ell$ with each homotopy class of curves appearing at most once in the sum and $k_i\in\mathbb{Z}$.

\begin{definition}
We will write $\mathcal{A}_{SL_2,S}(\mathbb{Z}^t)$ for the set of all $\ell\in\mathcal{A}_L(S,\mathbb{Z})$ that have integral coordinates.
\end{definition}

In addition to the laminations, we will be interested in a version of the moduli space of $PGL_2$-local systems on~$S$. To define this object, let $S'$ be the surface obtained by shrinking the boundary components of $S$ to punctures. If we equip this surface $S'$ with a complete, finite-area hyperbolic metric, then its universal cover can be identified with the hyperbolic plane $\mathbb{H}$. The punctures on $S'$ give rise to a set $\mathcal{F}_\infty(S)$ of points on the boundary $\partial\mathbb{H}$, and the action of~$\pi_1(S)$ by deck transformations on the universal cover gives rise to an action of~$\pi_1(S)$ on this set $\mathcal{F}_\infty(S)$.

\begin{definition}[\cite{IHES}, Lemma~1.1]
Consider the pair $(\rho,\psi)$ where $\rho:\pi_1(S)\rightarrow PGL_2(\mathbb{C})$ is a group homomorphism and $\psi:\mathcal{F}_\infty(S)\rightarrow\mathbb{P}^1(\mathbb{C})$ is a $(\pi_1(S),\rho)$-equivariant map from the set described above into $\mathbb{P}^1(\mathbb{C})$. That is, for any $\gamma\in\pi_1(S)$, we have 
\[
\psi(\gamma c)=\rho(\gamma)\psi(c).
\]
We will write $\mathcal{X}_{PGL_2,S}(\mathbb{C})$ for the space such pairs modulo the action of $PGL_2(\mathbb{C})$. A point of $\mathcal{X}_{PGL_2,S}(\mathbb{C})$ is called a \emph{framed $PGL_2(\mathbb{C})$-local system} on $S$.
\end{definition}

The space $\mathcal{X}_{PGL_2,S}(\mathbb{C})$ is the set of complex points of a moduli space $\mathcal{X}_{PGL_2,S}$.

\begin{proposition}[\cite{IHES}, Theorem~9.1]
For any ideal triangulation $T$, there are rational functions $X_i$~$(i\in I)$ on the space $\mathcal{X}_{PGL_2,S}$, which provide a birational map 
\[
\mathcal{X}_{PGL_2,S}\dashrightarrow\mathbb{G}_m^{|I|}.
\]
\end{proposition}

\subsection{The classical canonical map}
\label{sec:TheClassicalCanonicalMap}

In~\cite{IHES}, Fock and Goncharov describe the following method for calculating the monodromy of a framed $PGL_2(\mathbb{C})$-local system $m$ around an oriented loop $\ell$ on~$S$. To calculate this monodromy, let us first fix an ideal triangulation $T$, and let $X_i\in\mathbb{C}^*$ be the coordinate of~$m$ corresponding to the edge $i$ of~$T$. Choose a square root $X_i^{1/2}$ for~$X_i$.

Now deform $\ell$ so that it intersects each edge of the ideal triangulation $T$ in the minimal number of points. Suppose $i_1,\dots,i_s$ are the edges of $T$ that $\ell$ intersects in order (so an edge may appear more than once on this list). After crossing the edge $i_k$, the curve $\ell$ enters a triangle $t$ of~$T$ before leaving through the next edge. If the curve $\ell$ turns to the left before leaving $t$, then we form the matrix 
\[
M_k=
\left( \begin{array}{cc}
X_{i_k}^{1/2} & X_{i_k}^{1/2} \\
0 & X_{i_k}^{-1/2} \end{array} \right).
\]
On the other hand, if $\ell$ turns to the right before leaving $t$, then we form the matrix
\[
M_k=
\left( \begin{array}{cc}
X_{i_k}^{1/2} & 0 \\
X_{i_k}^{-1/2} & X_{i_k}^{-1/2} \end{array} \right).
\]
We can then multiply the matrices $M_k$ defined in this way to get a matrix representing the monodromy:
\[
\rho(\ell)=M_1\dots M_s.
\]

\begin{proposition}
\label{prop:monodromyperipheral}
Let $\ell$ be a peripheral loop on~$S$. Then $\rho(\ell)$ has eigenvalues $X_1^{a_1}\dots X_n^{a_n}$ and $X_1^{-a_1}\dots X_n^{-a_n}$ where $a_i$ is the coordinate of~$\ell$ associated to the edge~$i$.
\end{proposition}

\begin{proof}
Deform the loop so that it intersects each edge in the minimal number of points. We may assume $\ell$ is oriented so that it turns always to the right. If $i_1,\dots,i_n$ are the edges of~$T$ that $\ell$ crosses in order, then we can use the construction described above to compute the matrix $\rho(\ell)$. It is given by 
\[
\prod_{k=1}^s\left( \begin{array}{cc}
X_{i_k}^{1/2} & 0 \\
X_{i_k}^{-1/2} & X_{i_k}^{-1/2} \end{array} \right)
=\left( \begin{array}{cc}
X_{i_1}^{1/2}\dots X_{i_s}^{1/2} & 0 \\
C & X_{i_1}^{-1/2}\dots X_{i_s}^{-1/2} \end{array} \right)
\]
where $C$ is a polynomial in the variables $X_{i_k}^{\pm1/2}$. The eigenvalues of this matrix are the diagonal elements.
\end{proof}

We will now use this construction to associate a canonical function $\mathbb{I}(\ell)$ to a lamination $\ell\in\mathcal{A}_{SL_2,S}(\mathbb{Z}^t)$.

\Needspace*{2\baselineskip}
\begin{definition}[\cite{IHES}, Definition~12.4]\mbox{}
\label{def:classicalcanonical}
\begin{enumerate}
\item Let $\ell\in\mathcal{A}_L(S,\mathbb{Z})$ be a lamination consisting of a single nonperipheral curve of weight~$k$. Then the value of $\mathbb{I}(\ell)$ on~$m\in\mathcal{X}_{PGL_2,S}(\mathbb{C})$ is the trace of the $k$th power of the matrix~$\rho(\ell)$ constructed above.

\item Let $\ell\in\mathcal{A}_L(S,\mathbb{Z})$ be a lamination consisting of a single peripheral curve of weight $k$. The data of $m$ provide a distinguished eigenspace of $\rho(\ell)$ with eigenvalue $\lambda_\ell$, and we define the value of $\mathbb{I}(\ell)$ at $m$ to be $\lambda_\ell^k$.

\item Let $\ell\in\mathcal{A}_L(S,\mathbb{Z})$ and write $\ell=\sum_ik_i\ell_i$ where $\ell_i$ are the curves of $\ell$ with each homotopy class of curves appearing at most once in the sum and $k_i\in\mathbb{Z}$. Then
\[
\mathbb{I}(\ell)\coloneqq\prod_i\mathbb{I}(k_i\ell_i).
\]
\end{enumerate}
\end{definition}

Thus for any choice of square roots $X_i^{1/2}$, we define $\mathbb{I}(\ell)(m)$ as a certain Laurent polynomial in the $X_i^{1/2}$. Notice that if $\ell\in\mathcal{A}_{SL_2,S}(\mathbb{Z}^t)$, then the total weight of curves intersecting an edge $i$ of the triangulation is always even. It follows that in this case $\mathbb{I}(\ell)(m)$ is a Laurent polynomial in the variables $X_i$ which is independent of the choice of square roots. Thus we get a canonical map 
\[
\mathbb{I}:\mathcal{A}_{SL_2,S}(\mathbb{Z}^t)\rightarrow\mathbb{Q}(\mathcal{X}_{PGL_2,S}).
\]
Fock and Goncharov prove the following in~\cite{IHES}.

\begin{theorem}[\cite{IHES}, Theorem~12.2]
\label{thm:classicalproperties}
The canonical functions $\mathbb{I}(\ell)$ for $\ell\in\mathcal{A}_{SL_2,S}(\mathbb{Z}^t)$ satisfy the following properties:
\begin{enumerate}
\item For any choice of ideal triangulation, $\mathbb{I}(\ell)$ is a Laurent polynomial in the coordinates $X_i$ with highest term $X_1^{a_1}\dots X_n^{a_n}$ where $a_i$ is the coordinate of $\ell$ associated to the edge~$i$.

\item The coefficients of the Laurent polynomial $\mathbb{I}(\ell)$ are positive integers.

\item For any laminations $\ell$,~$\ell'\in\mathcal{A}_{SL_2,S}(\mathbb{Z}^t)$, we have 
\[
\mathbb{I}(\ell)\mathbb{I}(\ell')=\sum_{\ell''\in\mathcal{A}_{SL_2,S}(\mathbb{Z}^t)}c(\ell,\ell';\ell'')\mathbb{I}(\ell'')
\]
where $c(\ell,\ell';\ell'')$ are nonnegative integers and only finitely many terms are nonzero.
\end{enumerate}
\end{theorem}

If $\ell\in\mathcal{A}_L(S,\mathbb{Z})$, it is still true that the highest term of $\mathbb{I}(\ell)(m)$ is $X_1^{a_1}\dots X_n^{a_n}$, but in this case the $a_i$ may be only half integers.

\subsection{The Chekhov-Fock algebra}

We now define a quantum version of the space~$\mathcal{X}_{PGL_2,S}$. Let $T$ be an ideal triangulation of~$S$, and let $t_1,\dots,t_m$ be the triangles of $T$. Label the sides of $t_i$ by $e_{i1}$,~$e_{i2}$, and~$e_{i3}$ in such a way that these edge occur in the clockwise order around $t_i$.

\begin{definition}
The \emph{triangle algebra} $\mathcal{T}_{t_i}^q$ is an algebra over $\mathbb{C}$ generated by three elements $X_{i1}$,~$X_{i2}$, and~$X_{i3}$ and their inverses, subject to the relations 
\[
X_{i1}X_{i2}=q^2X_{i2}X_{i1}, \quad X_{i2}X_{i3}=q^2X_{i3}X_{i2}, \quad X_{i3}X_{i1}=q^2X_{i1}X_{i3}.
\]
\end{definition}

We think of the generator $X_{ij}$ of this definition as being associated to the side $e_{ij}$ of~$t_i$. We can now form the tensor product algebra $\bigotimes_{i=1}^m\mathcal{T}_{t_i}^q$. By convention, when describing an element $x_1\otimes\dots\otimes x_m$ of this algebra, we omit those factors $x_i$ that equal the identity element~1 in~$\mathcal{T}_{t_i}^q$. We now associate to each edge $e$ of the ideal triangulation an element of this tensor product as follows:
\begin{enumerate}
\item If $e$ separates two triangles $t_i$ and $t_j$, and if $X_{ia}\in\mathcal{T}_{t_i}^q$ and $X_{jb}\in\mathcal{T}_{t_j}^q$ are the generators associated to the sides of $t_i$ and $t_j$ corresponding to $e$, then we consider the element $X_e=X_{ia}\otimes X_{jb}$.

\item If $e$ corresponds to two sides of the same triangle~$t_i$, and if $X_{ia}$ and~$X_{ib}\in\mathcal{T}_{t_i}^q$ are the generators associated to these two sides and $X_{ia}$ is the side that comes first when going counterclockwise around their common vertex, then we consider the element $X_e=q^{-1}X_{ia}X_{ib}$.
\end{enumerate}

\begin{definition}
The \emph{Chekhov-Fock algebra} $\mathcal{X}_T^q$ is the subalgebra of $\bigotimes_{i=1}^m\mathcal{T}_{t_i}^q$ generated by the elements $X_e$ and their inverses.
\end{definition}

The generators of the Chekhov-Fock algebra satisfy the relations 
\[
X_iX_j=q^{2\varepsilon_{ij}}X_jX_i.
\]
One can also check that this algebra satisfies the Ore condition from ring theory, and hence we can form its noncommutative fraction field $\widehat{\mathcal{X}}_T^q$. The following result, originally due to Chekhov and Fock~\cite{CF}, provides a way of relating the division rings $\widehat{\mathcal{X}}_T^q$ for different ideal triangulations $T$.

\begin{theorem}
Let $T$ and $T'$ be two ideal triangulations of $S$. Then there exists an algebra isomorphism 
\[
\Phi_{TT'}^q:\widehat{\mathcal{X}}_{T'}^q\rightarrow\widehat{\mathcal{X}}_T^q
\]
such that $\Phi_{TT''}^q=\Phi_{TT'}^q\circ\Phi_{T'T''}^q$ for triangulations $T$, $T'$, and~$T''$.
\end{theorem}

In the present paper, we will also consider the Chekhov-Fock algebra associated to a fourth root $\omega=q^{1/4}$ of~$q$. To avoid confusion, we will denote this algebra by $\mathcal{Z}_T^\omega$ instead of~$\mathcal{X}_T^\omega$, and we will denote the generators of this algebra by $Z_i$ instead of $X_i$. There is a natural embedding of $\mathcal{X}_T^q$ into this algebra which maps the generator~$X_i$ to~$Z_i^2$. Thus we can think of $Z_i$ as a formal square root of $X_i$.

One would like to have maps relating the algebras $\mathcal{Z}_T^\omega$ for different ideal triangulations~$T$. As explained in~\cite{BonahonWong}, for each ideal triangulation $T$ there is a subalgebra $\widehat{\mathcal{Z}}_T^\omega\subseteq\widehat{\mathcal{X}}_T^\omega$ of the fraction field of $\mathcal{X}_T^\omega$ which is well behaved with respect to change of the triangulation, and one has the following result:

\begin{theorem}[\cite{BonahonWong}, Theorem~6]
Let $T$ and $T'$ be two ideal triangulations of~$S$. When $q=\omega^4$, there is an algebra isomorphism 
\[
\Theta_{TT'}^\omega:\widehat{\mathcal{Z}}_{T'}^\omega\rightarrow\widehat{\mathcal{Z}}_T^\omega
\]
extending the map $\Phi_{TT'}^q:\widehat{\mathcal{X}}_{T'}^q\rightarrow\widehat{\mathcal{X}}_T^q$ and having the property that $\Theta_{TT''}^\omega=\Theta_{TT'}^\omega\circ\Theta_{T'T''}^\omega$ for triangulations $T$, $T'$, and~$T''$.
\end{theorem}

Let us now consider the semiring $\mathbb{Z}_{\geq0}[\omega,\omega^{-1}]$ of Laurent polynomials in~$\omega$ with positive integral coefficients.

\begin{lemma}
\label{lem:quantize1}
If $c$ is an element of $\mathbb{Z}_{\geq0}[\omega,\omega^{-1}]$ whose classical limit is~1, then $c=\omega^m$ for some integer~$m$.
\end{lemma}

\begin{proof}
Since $c$ lies in $\mathbb{Z}_{\geq0}[\omega,\omega^{-1}]$, we can write $c=\sum_ic_i\omega^i$ where the sum is finite and $c_i\in\mathbb{Z}_{\geq0}$. Since we have $c=1$ in the classical limit, we know $\sum_ic_i=1$. Since each of the $c_i$ is a nonnegative integer, it follows that $c_m=1$ for some $m$, and all of the other $c_i$ vanish. Thus $c=\omega^m$ for some $m$.
\end{proof}

The following is a canonical way of forming the product of a collection of noncommutative factors in the Chekhov-Fock algebra.

\begin{definition}
If $x_1,\dots,x_n$ are elements of an algebra such that $x_ix_j=\omega^{2\varepsilon_{ij}}x_jx_i$, then the \emph{Weyl ordering} of the monomial $x_1\dots x_n$ is the monomial 
\[
[x_1\dots x_n]=\omega^{-\sum_{i<j}\varepsilon_{ij}}x_1\dots x_n.
\]
\end{definition}

The importance of this formula stems from the fact that it is invariant under all permutations of the $x_i$.

\begin{lemma}
\label{lem:Weyllimit}
Let $*$ be the canonical involutive antiautomorphism of $\mathcal{Z}_T^\omega$. If $c\cdot Z_1^{\mu_1}\dots Z_n^{\mu_n}$ with $c\in\mathbb{Z}_{\geq0}[\omega,\omega^{-1}]$ is a monomial which is $*$-invariant and has classical limit $Z_1^{\mu_1}\dots Z_n^{\mu_n}$, then $c\cdot Z_1^{\mu_1}\dots Z_n^{\mu_n}=[Z_1^{\mu_1}\dots Z_n^{\mu_n}]$.
\end{lemma}

\begin{proof}
Let $c\cdot Z_1^{\mu_1}\dots Z_n^{\mu_n}$ be an expression satisfying the hypotheses of the lemma. Then the classical limit of $c\in\mathbb{Z}_{\geq0}[\omega,\omega^{-1}]$ is~1, so by Lemma~\ref{lem:quantize1}, we can write $c=\omega^m$ for some integer~$m$. Since the expression is $*$-invariant, we have $*(c\cdot Z_1^{\mu_1}\dots Z_n^{\mu_n})=c\cdot Z_1^{\mu_1}\dots Z_n^{\mu_n}$, or equivalently
\[
\omega^{-m} Z_n^{\mu_n}\dots Z_1^{\mu_1}=\omega^m Z_1^{\mu_1}\dots Z_n^{\mu_n}.
\]
We have 
\begin{align*}
Z_n^{\mu_n}\dots Z_1^{\mu_1} &= \omega^{-2\sum_{j=2}^n\varepsilon_{1j}\mu_1\mu_j} Z_1^{\mu_1}(Z_n^{\mu_n}\dots Z_2^{\mu_2}) \\
&= \omega^{-2\sum_{j=2}^n\varepsilon_{1j}\mu_1\mu_j} \omega^{-2\sum_{j=3}^n\varepsilon_{2j}\mu_2\mu_j} Z_1^{\mu_1}Z_2^{\mu_2} (Z_n^{\mu_n}\dots Z_3^{\mu_3}) \\
&= \dots \\
&= \omega^{-2\sum_{i<j}\varepsilon_{ij}\mu_i\mu_j} Z_1^{\mu_1}\dots Z_n^{\mu_n},
\end{align*}
and thus $*$-invariance is equivalent to 
\[
\omega^{-m} \omega^{-2\sum_{i<j}\varepsilon_{ij}\mu_i\mu_j} Z_1^{\mu_1}\dots Z_n^{\mu_n} = \omega^m Z_1^{\mu_1}\dots Z_n^{\mu_n}.
\]
Equating coefficients, we find that $-m-2\sum_{i<j}\varepsilon_{ij}\mu_i\mu_j=m$, or equivalently 
\[
m=-\sum_{i<j}\varepsilon_{ij}\mu_i\mu_j.
\]
Hence 
\[
c\cdot Z_1^{\mu_1}\dots Z_n^{\mu_n}=\omega^{-\sum_{i<j}\varepsilon_{ij}\mu_i\mu_j} Z_1^{\mu_1}\dots Z_n^{\mu_n}
\]
is the Weyl ordering of the classical limit $Z_1^{\mu_1}\dots Z_n^{\mu_n}$.
\end{proof}

\subsection{The skein algebra}

In this section, we review the skein algebra of a surface, following~\cite{BonahonWong}.

\begin{definition}
\label{def:framedlink}
If $K$ is a one-dimensional submanifold of a three-dimensional manifold~$M$, then a \emph{framing} for~$K$ is a continuous map that assigns to each point $p\in K$ a vector $v_p\in T_pM$ such that $v_p$ is transverse to~$K$. For any surface $S$, a \emph{framed link} $K\subseteq S\times[0,1]$ is a 1-dimensional submanifold with framing such that 
\begin{enumerate}
\item $\partial K=K\cap\partial(S\times[0,1])$ consists of finitely many points in $(\partial S)\times[0,1]$.

\item At every point of~$\partial K$, the framing is vertical, that is, parallel to the $[0,1]$ factor and pointing towards~1.

\item For every component $\gamma$ of~$\partial S$, the points of~$\partial K$ that are in $\gamma\times[0,1]$ lie at different elevations, that is, at different points of the $[0,1]$ factor.
\end{enumerate}
An isotopy of framed links is required to respect all three conditions. 
\end{definition}

When drawing pictures of framed links projected to the surface, it is useful to adopt the conventions of Bonahon and Wong~\cite{BonahonWong}. Suppose $S$ is a decorated surface and $\gamma$ is a segment of $\partial S$ connecting adjacent marked points. If $K\subseteq S\times[0,1]$ is a framed link, then the points of $(\partial K)\cap(\gamma\times[0,1])$ are ordered by their elevation. To capture this information in pictures, we choose an arbitrary orientation of the edge $\gamma$. We then have two orderings of the set $(\partial K)\cap(\gamma\times[0,1])$. One is the order given by the elevations, and the other is given by the orientation of $\gamma$ if we identify each point of $(\partial K)\cap(\gamma\times[0,1])$ with its projection to~$\gamma$. After an isotopy of $K$, we can arrange that these two orderings coincide. We can then represent the link by drawing its projection to~$S$, and we can recover the ordering of $(\partial K)\cap(\gamma\times[0,1])$ by elevation using the chosen orientation of $\gamma$. Note that if we reverse the orientation of $\gamma$, then we must modify the projection of $K$ by applying a half twist as illustrated below:
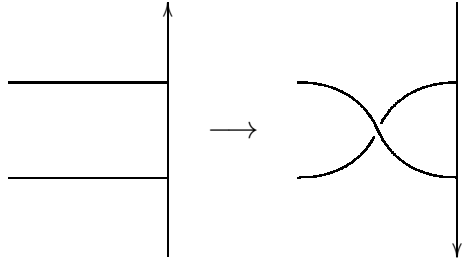
\[
\xy 0;/r.50pc/: 
(10,8)*{}="C";
(10,-8)*{}="C1";
(0,3)*{}="X1"; 
(10,3)*{}="Y1";
(0,-3)*{}="X2"; 
(10,-3)*{}="Y2";
{"C1"\ar@{->}"C"},
"X1";"Y1" **\crv{(0,3) & (10,3)};
"X2";"Y2" **\crv{(0,-3) & (10,-3)};
\endxy
\quad
\longrightarrow
\quad
\xy 0;/r.50pc/: 
(10,8)*{}="C";
(10,-8)*{}="C1";
(0,3)*{}="X"; 
(10,3)*{}="Y";
{"C"\ar@{->}"C1"},
\htwist~{(0,3)}{(10,3)}{(0,-3)}{(10,-3)}; 
\endxy
\]

\begin{definition}
We write $\mathcal{K}(S)$ for the vector space over~$\mathbb{C}$ generated by isotopy classes of framed links in $S\times[0,1]$. If $K_1$ and $K_2$ are elements of $\mathcal{K}(S)$, we write $K_1K_2$ for the element defined as the union of $K_1$ and $K_2$ with $K_1$ rescaled to lie in $S\times[0,\frac{1}{2}]$ and $K_2$ rescaled to lie in $S\times[\frac{1}{2},1]$. The space $\mathcal{K}(S)$ equipped with this operation is called the \emph{framed link algebra}.
\end{definition}

\begin{definition}
Three links $K_1$, $K_0$, and~$K_\infty$ form a \emph{Kauffman triple} if they coincide everywhere except over a small disk in~$S$ where they look like the pictures below equipped with the vertical framing:
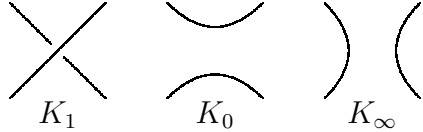
\[
\stackrel{
\xygraph{
    !{0;/r3.0pc/:}
    [u(0.5)]!{\xoverv}
}}
{K_1}
\qquad
\stackrel{
\xygraph{
    !{0;/r3.0pc/:}
    [u(0.5)]!{\xunoverh}
}}
{K_0}
\qquad
\stackrel{
\xygraph{
    !{0;/r3.0pc/:}
    [u(0.5)]!{\xunoverv}
}}
{K_\infty}
\]
The \emph{skein algebra} $\mathcal{S}^A(S)$ is the quotient of the framed link algebra $\mathcal{K}(S)$ by the two-sided ideal generated by all elements $K_1-A^{-1}K_0-AK_\infty$ where $(K_1,K_0,K_\infty)$ ranges over all Kauffman triples. An element $[K]\in\mathcal{S}^A(S)$ represented by a framed link $K$ is called a \emph{skein}.
\end{definition}

By construction, the elements of $\mathcal{S}^A(S)$ satisfy the \emph{skein relation}
\[
[K_1]=A^{-1}[K_0]+A[K_\infty]
\]
whenever $(K_0,K_1,K_\infty)$ is a Kauffman triple.

In what follows, we will often consider skeins equipped with additional data:

\begin{definition}
A \emph{state} for a skein $[K]\in\mathcal{S}^A(S)$ is a map $s:\partial K\rightarrow\{+,-\}$ which associates a sign $\pm$ to each point of~$\partial K$. A \emph{stated skein} is a skein equipped with a state, and we write $\mathcal{S}_{\mathrm{s}}^A(S)$ for the algebra generated by stated skeins.
\end{definition}

In addition to the superposition operation for multiplying skeins, there is a gluing operation that we can use to combine two skeins to get a third. Suppose that $S_1$ and $S_2$ are decorated surfaces equipped with ideal triangulations $T_1$ and $T_2$, respectively. Let $e_1\subseteq\partial S_1$ and $e_2\subseteq\partial S_2$ be boundary edges of the triangulations, and suppose that we glue the surfaces $S_1$ and $S_2$ along $e_1$ and $e_2$ to obtain a new decorated surface $S$ with an ideal triangulation~$T$. We allow the ``self-gluing'' case $S_1=S_2$ as long as the edges $e_1$ and $e_2$ are distinct.

Now suppose we are given skeins $[K_1]\in\mathcal{S}^A(S_1)$ and $[K_2]\in\mathcal{S}^A(S_2)$ such that $K_1\cap(e_1\times[0,1])$ and $K_2\cap(e_2\times[0,1])$ have the same number of points. By applying an isotopy of framed links, we can ensure that $K_1$ and $K_2$ glue together to give a framed link $K\subseteq S\times[0,1]$. This new link is well defined since we assumed that the points of $K_1\cap(e_1\times[0,1])$ and $K_2\cap(e_2\times[0,1])$ lie at different elevations. This link represents a skein $[K]\in\mathcal{S}^A(S)$ called the skein obtained by \emph{gluing} $[K_1]\in\mathcal{S}^A(S_1)$ and $[K_2]\in\mathcal{S}^A(S_2)$.

\begin{definition}
If the skein $[K]\in\mathcal{S}^A(S)$ is obtained by gluing $[K_1]\in\mathcal{S}^A(S_1)$ and $[K_2]\in\mathcal{S}^A(S_2)$, then the states $s:\partial K\rightarrow\{+,-\}$, $s_1:\partial K_1\rightarrow\{+,-\}$, and~$s_2:\partial K_2\rightarrow\{+,-\}$ are said to be \emph{compatible} if $s_1$ and $s_2$ coincide on $\partial K_1\cap(e_1\times[0,1])=\partial K_2\cap(e_2\times[0,1])$ for the identification given by the gluing and if $s$ coincides with the restrictions of $s_1$ and $s_2$.
\end{definition}

\subsection{Quantum traces}

In~\cite{BonahonWong}, Bonahon and Wong constructed a natural map from the skein algebra into the Chekhov-Fock algebra. In this section, we will review some of the details of this construction. These results will play a crucial role in our construction of $\widehat{\mathbb{I}}^q$.

\begin{theorem}[\cite{BonahonWong},~Theorem~11]
\label{thm:BWmain}
For $A=\omega^{-2}$, there is a unique family of algebra homomorphisms
\[
\Tr_T^\omega:\mathcal{S}_{\mathrm{s}}^A(S)\rightarrow\widehat{\mathcal{Z}}_T^\omega,
\]
defined for each decorated surface $S$ and each ideal triangulation $T$ of~$S$, such that 
\begin{enumerate}
\item If the surface $S$ is obtained by gluing $S_1$ and $S_2$, if the ideal triangulation $T$ is obtained by gluing the ideal triangulations $T_1$ of~$S_1$ and~$T_2$ of~$S_2$, and if the skein $[K]\in\mathcal{S}^A(S)$ is obtained by gluing $[K_1]\in\mathcal{S}^A(S_1)$ and $[K_2]\in\mathcal{S}^A(S_2)$, then 
\[
\Tr_T^\omega([K,s])=\sum_{\text{compatible }s_1,s_2}\Tr_{T_1}^\omega([K_1,s_1])\otimes\Tr_{T_2}^\omega([K_2,s_2])
\]
where the sum is over all states $s_1$ and $s_2$ that are compatible with $s$ and with each other. Similarly, if the surface $S$, the ideal triangulation $T$, and the skein $[K]$ are obtained by gluing $S_1$, $T_1$, and~$[K_1]$, respectively, to themselves, then 
\[
\Tr_T^\omega([K,s])=\sum_{\text{compatible }s_1}\Tr_{T_1}^\omega([K_1,s_1]).
\]

\item\label{part:trianglecase} If $S$ is a triangle and $K$ projects to the following single arc embedded in~$S$ with vertical framing at every point, 
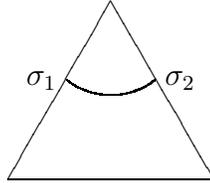
\[
\xy /l1.25pc/:
{\xypolygon3"A"{~:{(-3,0):}}},
(-0.1,-1)*{}="j1";
(2.1,-1)*{}="i1";
(2.75,-1)*{\sigma_1};
(-0.75,-1)*{\sigma_2};
"j1";"i1" **\crv{(0.5,-0.5) & (1.5,-0.5)};
\endxy
\]
then we have 
\[
\Tr_T^\omega([K,s])=
\begin{cases}
0 & \mbox{if } \sigma_1=-, \sigma_2=+ \\
[Z_1^{\sigma_1}Z_2^{\sigma_2}] & \mbox{otherwise}
\end{cases}
\]
where $Z_1$ and $Z_2$ are the generators of the triangle algebra associated to the sides of $S$ that intersect the arc with states $\sigma_1$ and $\sigma_2$, respectively.

\item If $S$ is a triangle and $K$ projects to the following single arc embedded in~$S$ with vertical framing at every point, 
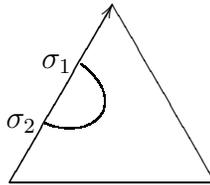
\[
\xy /l1.25pc/:
{\xypolygon3"A"{~:{(-3,0):}}},
{"A2"\ar@{->}"A1"},
(1.8,-1.5)*{}="j1";
(2.7,0)*{}="i1";
(2.4,-1.5)*{\sigma_1};
(3.3,0)*{\sigma_2};
"j1";"i1" **\crv{(0.5,-0.5) & (1.5,0.5)};
\endxy
\]
then we have 
\[
\Tr_T^\omega([K,s])=
\begin{cases}
0 & \mbox{if } \sigma_1=\sigma_2 \\
-\omega^{-5} & \mbox{if } \sigma_1=+, \sigma_2=- \\
\omega^{-1} & \mbox{if } \sigma_1=-, \sigma_2=+.
\end{cases}
\]
\end{enumerate}
\end{theorem}

Recall that in Section~\ref{sec:TheClassicalCanonicalMap} we described how to construct for any oriented loop $\ell$ on~$S$ a matrix $\rho(\ell)$ whose entries are Laurent polynomials in the square roots $X_i^{1/2}$. If we replace each $X_i^{1/2}$ in this construction by the corresponding generator $Z_i$ of the Chekhov-Fock algebra~$\mathcal{Z}_T^1$, then we get a matrix whose entries are elements of~$\mathcal{Z}_T^1$, and we denote this matrix again by~$\rho(\ell)$. In particular, if $[K]$ is a connected skein with $\partial K=\varnothing$, then we can apply this construction to its projection and take the trace to get an invariant denoted~$\Tr_T^1([K])$.

\begin{theorem}[\cite{BonahonWong}, Theorem~1]
\label{thm:BWclassical}
When $\omega=1$ the map $\Tr_T^\omega:\mathcal{S}^A(S)\rightarrow\widehat{\mathcal{Z}}_T^\omega$ coincides with $\Tr_T^1:\mathcal{S}^1(S)\rightarrow\mathcal{Z}_T^1$.
\end{theorem}

The map $\Tr_T^\omega$ provided by Theorem~\ref{thm:BWmain} depends on the choice of ideal triangulation $T$, but Bonahon and Wong show that it is well behaved with respect to changes of triangulation. More precisely, this map has image in the algebra $\widehat{\mathcal{Z}}_T^\omega$, and they prove the following result.

\begin{theorem}[\cite{BonahonWong}, Theorem~2]
\label{thm:BWchangetriangulation}
Let $T$ and $T'$ be two ideal triangulations of~$S$, and let $[K]\in\mathcal{S}^A(S)$. Then the coordinate change map 
\[
\Theta_{TT'}^\omega:\widehat{\mathcal{Z}}_{T'}^\omega\rightarrow\widehat{\mathcal{Z}}_T^\omega
\]
takes the Laurent polynomial $\Tr_{T'}^\omega([K])$ to the Laurent polynomial $\Tr_T^\omega([K])$.
\end{theorem}

We will now review the main ideas involved in the construction of $\Tr_T^\omega$. In~\cite{BonahonWong}, Bonahon and Wong define a \emph{split ideal triangulation} $\widehat{T}$ to be the cell decomposition of a surface obtained from an ideal triangulation $T$ by replacing each edge $i$ by two parallel copies of it, separated by a biangle.

\begin{lemma}[\cite{BonahonWong}, Lemma~23]
Let $K$ be a framed link in $S\times[0,1]$, and let $\widehat{T}$ be a split ideal triangulation of~$S$. Then $K$ can be deformed by an isotopy so that it satisfies the following conditions.
\begin{enumerate}
\item For every edge $i$ of $\widehat{T}$, $K$ is transverse to $i\times[0,1]$.

\item For every triangle $t_j$ of $\widehat{T}$, the intersection $K\cap(t_j\times[0,1])$ consists of finitely many disjoint arcs, each of which is contained in a constant elevation surface and joins two distinct components of $\partial t_j\times[0,1]$.

\item For every triangle $t_j$ of $\widehat{T}$, the components of $K\cap(t_j\times[0,1])$ lie at different elevations, and their framings are vertical.
\end{enumerate}
\end{lemma}

\begin{definition}
When a link $K$ satisfies the conditions in this lemma, we say that it is in \emph{good position} with respect to the split ideal triangulation $\widehat{T}$.
\end{definition}

Suppose that $B$ is a biangle whose sides are oriented in the same direction and $K$ is a framed link in $B\times[0,1]$ whose projection to $B$ has no crossing when drawn according to our picture conventions. Let $s$ be a state for $[K]$. In~\cite{BonahonWong}, Bonahon and Wong associate a number $\Tr_B^\omega([K,s])$ to these data. To calculate this number, let $a_{\sigma_2}^{\sigma_1}$ be the number of components of the type illustrated on the left below. Let $b_{\sigma_2}^{\sigma_1}$ be the number of components of the type illustrated in the middle and $c_{\sigma_2}^{\sigma_1}$ the number of the type illustrated on the right. Let $d$ be the number of closed components.
\[
\xy 0;/r.50pc/: 
(5,10)*{B}; 
(0,12)*{}="B"; 
(10,12)*{}="C";
(0,-6)*{}="B1"; 
(10,-6)*{}="C1";
(0,3)*{}="X"; 
(10,3)*{}="Y";
{"B1"\ar@{->}"B"},
{"C1"\ar@{->}"C"},
"X";"Y" **\crv{(4,1) & (6,5)};
(-1.5,3)*{\sigma_1}; 
(11.5,3)*{\sigma_2}; 
\endxy
\qquad
\qquad
\xy 0;/r.50pc/: 
(5,10)*{B}; 
(0,12)*{}="B"; 
(10,12)*{}="C";
(0,-6)*{}="B1"; 
(10,-6)*{}="C1";
(0,6)*{}="X"; 
(0,0)*{}="Y";
{"B1"\ar@{->}"B"},
{"C1"\ar@{->}"C"},
"X";"Y" **\crv{(4,6) & (4,0)};
(-1.5,6)*{\sigma_1}; 
(-1.5,0)*{\sigma_2}; 
\endxy
\qquad
\qquad
\xy 0;/r.50pc/: 
(5,10)*{B}; 
(0,12)*{}="B"; 
(10,12)*{}="C";
(0,-6)*{}="B1"; 
(10,-6)*{}="C1";
(10,6)*{}="X"; 
(10,0)*{}="Y";
{"B1"\ar@{->}"B"},
{"C1"\ar@{->}"C"},
"X";"Y" **\crv{(6,6) & (6,0)};
(11.5,6)*{\sigma_1}; 
(11.5,0)*{\sigma_2}; 
\endxy
\]

\begin{definition}[\cite{BonahonWong}, Lemma~20]
\label{def:computebiangle}
For a stated skein $[K,s]\in\mathcal{S}_{\mathrm{s}}^A(B)$ with no crossing, let $a_{\sigma_2}^{\sigma_1}$, $b_{\sigma_2}^{\sigma_1}$, $c_{\sigma_2}^{\sigma_1}$, and $d$ be defined as above. If one of the numbers $a_-^+$, $a_+^-$, $b_+^+$, $b_-^-$, $c_+^+$, or~$c_-^-$ is nonzero, then $\Tr_B^\omega([K,s])\coloneqq0$. Otherwise,
\[
\Tr_B^\omega([K,s])\coloneqq(-1)^{b_-^++c_+^-}\omega^{-(5b_-^++b_+^-)}\omega^{5c_+^-+c_-^+}(-\omega^4-\omega^{-4})^d.
\]
\end{definition}

More generally, if $[K,s]\in\mathcal{S}_{\mathrm{s}}^A(B)$ is a stated skein represented by a link $K$ with the vertical framing whose projection to $B$ has crossings, then we first use the skein relations to write $[K]=\sum_{i=1}^N\omega^{2p_i}[K_i]$ where $N$ and $p_i$ are integers and each $K_i$ is a link with no crossings, equipped with the vertical framing. Then we define 
\[
\Tr_B^\omega([K,s])\coloneqq\sum_{i=1}^N\omega^{2p_i}\Tr_B^\omega([K_i,s]).
\]
In this way, we associate a number $\Tr_{B}([K,s])$ to any stated skein $[K,s]\in\mathcal{S}_{\mathrm{s}}^A(B)$.

Now suppose $K$ is a framed link in $S\times[0,1]$ which is in good position with respect to a split ideal triangulation $\widehat{T}$. For each triangle $t_j$ of $\widehat{T}$, we put $K_j=K\cap(t_j\times[0,1])$, and for each biangle $B_i$, we put $L_i=K\cap(B_i\times[0,1])$. Let $k_1,\dots,k_l$ be the components of $K_j$ in order of increasing elevation, and define 
\[
\Tr_{t_j}^\omega([K_j,s_j])\coloneqq\Tr_{t_j}^\omega([k_1,s_j])\Tr_{t_j}^\omega([k_2,s_j])\dots \Tr_{t_j}^\omega([k_l,s_j])
\]
where $\Tr_{t_j}^\omega([k_i,s_j])$ is defined as in part~\ref{part:trianglecase} of Theorem~\ref{thm:BWmain}. Then $\Tr_T^\omega([K,s])$ is defined as 
\[
\Tr_T^\omega([K,s])\coloneqq\sum_{\text{compatible }\sigma_j, \tau_i}\prod_{i=1}^n\Tr_{B_i}^\omega([L_i,\tau_i])\bigotimes_{j=1}^m\Tr_{t_j}^\omega([K_j,\sigma_j]).
\]
where the sum is over all states $\sigma_j:\partial K_j\rightarrow\{+,-\}$ and $\tau_i:\partial L_i\rightarrow\{+,-\}$ that are compatible with $s:\partial K\rightarrow\{+,-\}$ and with each other. It follows from~\cite{BonahonWong},~Proposition~26 that this expression $\Tr_T^\omega([K,s])$ depends only on the isotopy class of $K$ and the state~$s$.

\begin{lemma}
\label{lem:BWstar}
Let $*$ be the canonical involutive antiautomorphism of $\mathcal{Z}_T^\omega$. If $K\subseteq S\times[0,1]$ is a link with the vertical framing and no crossings, then $*\Tr_T^\omega([K])=\Tr_T^\omega([K])$.
\end{lemma}

\begin{proof}
Each triangle algebra $\mathcal{T}_{t_j}^\omega$ has an involutive antiautomorphism $*$ defined on generators by the formulas $*Z_{ji}=Z_{ji}$ and $*\omega=\omega^{-1}$. Define an operation $*$ on $\bigotimes_{j=1}^m\mathcal{T}_{t_j}^\omega$ by 
\[
*(x_1\otimes\dots\otimes x_m)=(*x_1)\otimes\dots\otimes(*x_m).
\]
This operation is an antihomomorphism. Moreover, for any generator $Z_i$ of the Chekhov-Fock algebra we have $*Z_i=Z_i$, so this operation $*$ induces the familiar antiautomorphism on~$\mathcal{Z}_T^\omega$. Now $\Tr_T^\omega([K])$ is obtained by summing the products $\prod_{i=1}^n\Tr_{B_i}^\omega([L_i,\tau_i])\bigotimes_{j=1}^m\Tr_{t_j}^\omega([K_j,\sigma_j])$ where $\Tr_{t_j}^\omega([K_j,\sigma_j])$ is a product 
\[
\Tr_{t_j}^\omega([K_j,\sigma_j])=\Tr_{t_j}^\omega([k_1,\sigma_j])\dots\Tr_{t_j}^\omega([k_l,\sigma_j])
\]
of Weyl ordered expressions $\Tr_{t_j}^\omega([k_i,\sigma_j])=[Z_{j1}^{\sigma_1}Z_{j2}^{\sigma_2}]$. These last expressions are $*$-invariant, so the $*$-operator on $\mathcal{T}_{t_i}^\omega$ acts by reversing the order of the factors in the product $\Tr_{t_j}^\omega([K_j,\sigma_j])$. It acts on the product $\prod_{i=1}^n\Tr_{B_i}^\omega([L_i,\tau_i])$ by sending $\omega$ to $\omega^{-1}$. The quantum trace $\Tr_T^\omega([K])$ transforms in the same way when we apply the elevation-reversing map to $K$. In other words, we have $*\Tr_T^\omega([K])=\Tr_T^\omega([K'])$ where $K'$ is obtained from $K$ by applying the map
\[
S\times[0,1]\rightarrow S\times[0,1]
\]
given by 
\[
(x,t)\mapsto(x,1-t).
\]
If $K$ has no crossings, then these links $K$ and $K'$ are isotopic, so we have $\Tr_T^\omega([K'])=\Tr_T^\omega([K])$. This completes the proof.
\end{proof}

\section{The quantum canonical map}
\label{sec:TheQuantumCanonicalMap}

\subsection{Definition of the map}

Let $S$ be a punctured surface that admits an ideal triangulation $T$. In this section, we define an element $\mathbb{I}^\omega(\ell)\in\mathcal{Z}_T^\omega$ for any lamination $\ell\in\mathcal{A}_L(S,\mathbb{Z})$ and prove that our definition is independent of all choices. We begin by defining $\mathbb{I}^\omega(\ell)$ for certain special laminations $\ell$, and then we extend the definition to arbitrary laminations by multiplicativity.

\begin{definition}
\label{def:peripheral}
Let $\ell$ be a lamination consisting of a single peripheral loop on~$S$. For an ideal triangulation $T$ of~$S$, we define 
\[
\mathbb{I}_T^\omega(\ell)\coloneqq[Z_1^{\mu_1}\dots Z_n^{\mu_n}]
\]
where $\mu_i=2a_i$ is twice the coordinate of~$\ell$ associated to the edge~$i$.
\end{definition}

In this definition, we had to choose an ideal triangulation $T$ of~$S$ and a labeling of the edges so that $Z_i$ is the generator of the Chekhov-Fock algebra associated to the edge $i$. We now show that $\mathbb{I}_T^\omega(\ell)$ is well behaved under changes of the triangulation.

\begin{lemma}
\label{lem:peripheralhighest}
If $\ell$ is a lamination consisting of a single peripheral loop of weight~1 on~$S$, and $K$ is a link in $S\times[0,1]$ with constant elevation and the vertical framing that projects to $\ell$, then $\mathbb{I}_T^\omega(\ell)$ is the highest term of $\Tr_T^\omega([K])$.
\end{lemma}

\begin{proof}
The highest term of $\Tr_T^\omega([K])$ arises in Bonahon and Wong's state sum as the term corresponding to the state that assigns $+$ to every point. This term is the unique highest term, and its coefficient lies in $\mathbb{Z}_{\geq0}[\omega,\omega^{-1}]$. By Lemma~\ref{lem:BWstar}, $\Tr_T^\omega([K])$ is $*$-invariant. When we apply the map $*$ to $\Tr_T^\omega([K])$, the highest term maps to another term with the same total degree, so this highest term must be mapped to itself. In other words, the highest term is $*$-invariant. By~Proposition~\ref{prop:monodromyperipheral}, the classical limit of the highest term is $Z_1^{\mu_1}\dots Z_n^{\mu_n}$, and therefore the highest term of the quantum trace $\Tr_T^\omega([K])$ is $[Z_1^{\mu_1}\dots Z_n^{\mu_n}]$ by Lemma~\ref{lem:Weyllimit}.
\end{proof}

\begin{lemma}
Let $\ell$ be a lamination consisting of a single peripheral loop on~$S$. If $T$ and $T'$ are ideal triangulations of $S$, then $\Theta_{TT'}^\omega\mathbb{I}_{T'}^\omega(\ell)=\mathbb{I}_T^\omega(\ell)$.
\end{lemma}

\begin{proof}
Suppose the loop $\ell$ has weight~1. Let $K$ be a link in $S\times[0,1]$ with constant elevation and the vertical framing that projects to $\ell$. By Lemma~\ref{lem:peripheralhighest}, the expression $\mathbb{I}_T^\omega(\ell)$ is the highest term of $\Tr_T^\omega([K])$, which arises in the state sum from the state that assigns $+$ to every point. Similarly, $\mathbb{I}_{T'}^\omega(\ell)$ is the term in $\Tr_{T'}^\omega([K])$ corresponding to the state that assigns $+$ to every point. By the proof  in~\cite{BonahonWong} of our Theorem~\ref{thm:BWchangetriangulation}, these two expressions are related by the coordinate change map.

This proves the lemma in the special case where $\ell$ consists of a single peripheral loop of weight~1. For such a lamination, we have $\mathbb{I}^\omega(k\ell)=\mathbb{I}^\omega(\ell)^k$. The lemma follows from this and the fact that the coordinate change map is an algebra homomorphism.
\end{proof}

\begin{definition}
\label{def:nonperipheral}
Let $\ell$ be a lamination consisting of a single nonperipheral curve of weight~1 on~$S$. Let $T$ be any ideal triangulation of~$S$. Then 
\[
\mathbb{I}_T^\omega(\ell)\coloneqq\Tr_T^\omega([K])
\]
where $K$ is a framed link in $S\times[0,1]$ with constant elevation and vertical framing that projects to $\ell$.
\end{definition}

By Theorem~\ref{thm:BWchangetriangulation}, we have $\Theta_{TT'}^\omega\mathbb{I}_{T'}^\omega(\ell)=\mathbb{I}_T^\omega(\ell)$ for ideal triangulations $T$ and $T'$. We therefore write $\mathbb{I}^\omega=\mathbb{I}_T^\omega$ when there is no possibility of confusion.

To define the quantum canonical map for a lamination consisting of a curve of weight $k>1$, we need the following special polynomials.

\begin{definition}
The \emph{Chebyshev polynomials} $F_k(t)\in\mathbb{Z}[t]$ are defined by $F_0(t)=2$, $F_1(t)=t$, and the recursion relation 
\[
F_{k+1}(t)=F_k(t)\cdot t-F_{k-1}(t)
\]
for $k\geq1$.
\end{definition}

The polynomials $T_k(t)$ defined by $T_0(t) = 1$, $T_1(t) = t$, $T_{k+1}(t) = 2t T_k(t) - T_{k-1}(t)$ are what are usually called \emph{Chebyshev polynomials of the first kind}. They are related to our Chebyshev polynomials $F_k(t)$ by $F_k(t) = 2 T_k(t/2)$. For us, the most important fact about the polynomials $F_k(t)$ is the following, which relates them to traces of $2\times 2$-matrices:

\begin{proposition}
\label{prop:Chebyshevtrace}
For any $M\in SL_2(\mathbb{C})$ and any nonnegative integer $k$, we have 
\[
\Tr(M^k)=F_k(\Tr(M)).
\]
\end{proposition}

\begin{proof}
The statement is clear for $k=0$,~$1$. Assume inductively that it holds for all $k\leq N$. A straightforward calculation shows that for any $2\times2$ matrix $A$ and any~$B\in SL_2(\mathbb{C})$, we have the identity 
\[
\Tr(A)\Tr(B)=\Tr(AB)+\Tr(AB^{-1}).
\]
Applying this identity with $A=M^N$ and $B=M$, we find $\Tr(M^N)\Tr(M)=\Tr(M^{N+1})+\Tr(M^{N-1})$, or equivalently, 
\begin{align*}
\Tr(M^{N+1}) &= \Tr(M^N)\Tr(M)-\Tr(M^{N-1}) \\
&= F_N(\Tr(M))\Tr(M)-F_{N-1}(\Tr(M)) \\
&= F_{N+1}(\Tr(M)).
\end{align*}
The proposition follows by induction.
\end{proof}

\begin{corollary}
\label{cor:classicalChebyshev}
Let $\ell$ be a lamination consisting of a single nonperipheral curve of weight~1, and let $k$ be any positive integer. Then 
\[
\mathbb{I}(k\ell)=F_k(\mathbb{I}(\ell)).
\]
\end{corollary}

\begin{proof}
Recall that $\mathbb{I}(\ell)$ is equal to the trace of a matrix $M\in SL_2(\mathbb{C})$, while $\mathbb{I}(k\ell)$ is equal to the trace of $M^k$. By~Proposition~\ref{prop:Chebyshevtrace}, we have $\Tr(M^k)=F_k(\Tr(M))$.
\end{proof}

This result motivates the definition of the quantum canonical map for loops with weight an integer $k>1$.

\begin{definition}
\label{def:multicurve}
Let $\ell$ be a lamination consisting of a single nonperipheral curve of weight~1, and let $k$ be a positive integer. Then we define 
\[
\mathbb{I}^\omega(k\ell)\coloneqq F_k(\mathbb{I}^\omega(\ell))
\]
where $F_k$ is the $k$th Chebyshev polynomial.
\end{definition}

A priori, one might expect the correct definition of $\mathbb{I}^\omega$ to involve some deformation of the Chebyshev polynomial $F_k$ with coefficients in $\mathbb{Z}[\omega,\omega^{-1}]$. Definition~\ref{def:multicurve} appears however to be an almost unique definition of $\mathbb{I}^\omega$ having certain desirable properties. A more thorough investigation of the possible uniqueness of our construction is a problem for future research.

We now wish to extend the definition of the quantum canonical map to all laminations by multiplicativity. To show that the result is a well defined element of the Chekhov-Fock algebra, we will need the following two lemmas.

\begin{lemma}
\label{lem:peripheralcentral}
If $\ell$ is a lamination consisting of a single peripheral curve, then $\mathbb{I}^\omega(\ell)$ is a central element of the Chekhov-Fock algebra.
\end{lemma}

\begin{proof}
Suppose the curve of $\ell$ has weight~1. Let $p$ denote the puncture to which this curve retracts, and let $j_1,\dots,j_r$ be the edges of the triangulation that end at $p$. (An edge appears twice on this list if both of its endpoints are the puncture $p$.) We claim that for any edge $i$ of the ideal triangulation, we have 
\[
\sum_{s=1}^r\varepsilon_{ij_s}=0.
\]
Indeed, if $i$ is an edge of the ideal triangulation with the property that $\varepsilon_{ij_s}=0$ for all $s\in\{1,\dots,r\}$, then this equation clearly holds. On the other hand, if we have $\varepsilon_{ij_s}\neq0$ for some $j_s$, then one can check that the number of edges $j_s$ that lie immediately to the left of~$i$ is the same as the number of edges that lie immediately to the right. It follows that the claim is true in this case also. Hence 
\begin{align*}
Z_i(Z_{j_1}\dots Z_{j_r}) &= \omega^{2\sum_s\varepsilon_{ij_s}}(Z_{j_1}\dots Z_{j_r})Z_i \\
&= (Z_{j_1}\dots Z_{j_r})Z_i,
\end{align*}
so the generator $Z_i$ commutes with $\mathbb{I}^\omega(\ell)=[Z_{j_1}\dots Z_{j_r}]$. Hence it commutes with $\mathbb{I}^\omega(k\ell)=\mathbb{I}^\omega(\ell)^k$.
\end{proof}

\begin{lemma}
\label{lem:disjointcommute}
If $\ell$ and $\ell'$ are disjoint simple closed curves on $S$, and if $k$ and $k'$ are integer weights for $\ell$ and $\ell'$, then $\mathbb{I}^\omega(k\ell)$ and $\mathbb{I}^\omega(k'\ell')$ commute.
\end{lemma}

\begin{proof}
By Lemma~\ref{lem:peripheralcentral}, it suffices to check this when $\ell$ and $\ell'$ are both nonperipheral. In this case, we have $\mathbb{I}^\omega(\ell)=\Tr_T^\omega([K])$ and $\mathbb{I}^\omega(\ell')=\Tr_T^\omega([K'])$ where $K$ and $K'$ are framed links with constant elevation and vertical framing that project to $\ell$ and $\ell'$, respectively. Since the quantum trace map is an algebra homomorphism, it follows that 
\begin{align*}
\mathbb{I}^\omega(\ell)\mathbb{I}^\omega(\ell') &= \Tr_T^\omega([K])\Tr_T^\omega([K']) \\
&= \Tr_T^\omega([K][K']) \\
&= \Tr_T^\omega([K'][K]) \\
&= \Tr_T^\omega([K'])\Tr_T^\omega([K]) \\
&= \mathbb{I}^\omega(\ell')\mathbb{I}^\omega(\ell).
\end{align*}
By applying Definition~\ref{def:multicurve}, we can express $\mathbb{I}^\omega(k\ell)$ and $\mathbb{I}^\omega(k'\ell')$ as polynomials in $\mathbb{I}^\omega(\ell)$ and $\mathbb{I}^\omega(\ell')$, respectively. It follows that these expressions commute.
\end{proof}

The following is the main definition of the present paper:

\begin{definition}
\label{def:canonicalmap}
Let $\ell$ be any lamination in $\mathcal{A}_L(S,\mathbb{Z})$ and write $\ell=\sum_ik_i\ell_i$ where $\ell_i$ are the curves of~$\ell$ with each homotopy class of curves appearing at most once in the sum and $k_i\in\mathbb{Z}$. Then
\[
\mathbb{I}^\omega(\ell)\coloneqq\prod_i\mathbb{I}^\omega(k_i\ell_i).
\]
In particular, the value of $\mathbb{I}^\omega$ on the empty lamination is the identity.
\end{definition}

By Lemma~\ref{lem:disjointcommute}, this product is independent of the order of the factors. Let us examine the classical limit of this expression.

\begin{proposition}
\label{prop:classicallimit}
For any lamination $\ell$, the classical limit of $\mathbb{I}^\omega(\ell)$ is identified with $\mathbb{I}(\ell)$ where the generator $Z_i$ is identified with $X_i^{1/2}$.
\end{proposition}

\begin{proof}
If $\ell$ is a lamination consisting of a single peripheral curve on~$S$, then we have $\mathbb{I}^1(\ell)=\mathbb{I}(\ell)$ by Definitions~\ref{def:classicalcanonical} and~\ref{def:peripheral}, Proposition~\ref{prop:monodromyperipheral}, and Theorem~\ref{thm:classicalproperties}. If $\ell$ is a lamination consisting of a single nonperipheral curve of weight~1, then $\mathbb{I}^1(\ell)=\Tr_T^1([K])=\mathbb{I}(\ell)$ by Theorem~\ref{thm:BWclassical}. Moreover, if $k$ is any positive integer, then we have 
\begin{align*}
\mathbb{I}^1(k\ell) &= F_k(\mathbb{I}^1(\ell)) \\
&= F_k(\mathbb{I}(\ell)) \\
&= \mathbb{I}(k\ell)
\end{align*}
by the previous remarks and Corollary~\ref{cor:classicalChebyshev}. The proposition now follows from~Definition~\ref{def:canonicalmap}.
\end{proof}

Recall that $*$ denotes the canonical involutive antiautomorphism of $\mathcal{Z}_T^\omega$. 

\begin{proposition}
\label{prop:starinvariance}
For any lamination~$\ell\in\mathcal{A}_L(S,\mathbb{Z})$, we have $*\mathbb{I}^\omega(\ell)=\mathbb{I}^\omega(\ell)$.
\end{proposition}

\begin{proof}
By~Lemma~\ref{lem:BWstar}, the element $\mathbb{I}^\omega(\ell)$ is $*$-invariant whenever $\ell$ is a lamination consisting of a single nonperipheral curve of weight~1. For any positive integer~$k$, we can express $\mathbb{I}^\omega(k\ell)$ as a polynomial in $\mathbb{I}^\omega(\ell)$ with integral coefficients, so this element $\mathbb{I}^\omega(k\ell)$ is $*$-invariant as well. Finally, for a lamination $\ell$ consisting of a single peripheral curve, the element $\mathbb{I}^\omega(\ell)$ is $*$-invariant by inspection.
\end{proof}

If $F$ is an element of $\mathcal{Z}_T^\omega$ having a highest term, then we will write $[F]^H$ for this highest term. Recall that by Theorem~\ref{thm:classicalproperties}, the Laurent polynomial $\mathbb{I}^1(\ell)$ has highest term of the form $Z_1^{\mu_1}\dots Z_n^{\mu_n}$ where $\mu_i=2a_i$ is twice the coordinate associated to the edge~$i$. We now examine the highest term of the quantum analog $\mathbb{I}^\omega(\ell)$.

\begin{proposition}
\label{prop:highest}
Let $\ell\in\mathcal{A}_L(S,\mathbb{Z})$ be a lamination on~$S$. Then the unique highest term of $\mathbb{I}^\omega(\ell)$ is 
\[
[\mathbb{I}^\omega(\ell)]^H=[Z_1^{\mu_1}\dots Z_n^{\mu_n}]
\]
where $\mu_i=2a_i$ is twice the coordinate of~$\ell$ associated to the edge~$i$.
\end{proposition}

\begin{proof}
If $\ell$ is a lamination consisting of a single curve, then the expression $\mathbb{I}^\omega(\ell)$ has a unique highest term, and its coefficient lies in $\mathbb{Z}_{\geq0}[\omega,\omega^{-1}]$. Thus the same is true of $\mathbb{I}^\omega(\ell)$ for any lamination. By Proposition~\ref{prop:starinvariance}, $\mathbb{I}^\omega(\ell)$ is $*$-invariant. When we apply the map $*$ to $\mathbb{I}^\omega(\ell)$, the highest term maps to another term with the same total degree, so this highest term must be mapped to itself. In other words, the highest term is $*$-invariant. By Lemma~\ref{lem:Weyllimit}, this term equals the Weyl ordering of its classical limit, namely $[Z_1^{\mu_1}\dots Z_n^{\mu_n}]$.
\end{proof}

This result has a consequence that will be useful later.

\begin{lemma}
\label{lem:highestdetermines}
If $\ell$ and~$\ell'$ are distinct laminations, then the two monomials obtained by taking the classical limits of the highest terms of $\mathbb{I}^\omega(\ell)$ and $\mathbb{I}^\omega(\ell')$ are distinct.
\end{lemma}

\begin{proof}
Proposition~\ref{prop:highest} says that the classical limits of the highest terms of these expressions are $Z_1^{\mu_1}\dots Z_n^{\mu_n}$ and $Z_1^{\mu_1'}\dots Z_n^{\mu_n'}$ where $\mu_i=2a_i$, $\mu_i'=2a_i'$, and $a_i$ and $a_i'$ denote the coordinates of~$\ell$ and~$\ell'$, respectively. This proves our claim since two laminations are equal if and only if their coordinates coincide.
\end{proof}

In~\cite{BWrep}, Bonahon and Wong studied their quantum trace construction in the special case where $A$ is a root of unity and proved the following result.

\begin{theorem}[\cite{BWrep}, Theorem~21]
\label{thm:BWFrobenius}
Let $S$ be a punctured surface with no boundary, and let $T$ be an ideal triangulation of $S$. If $A^4$ is a primitive $N$th root of unity, $A=\omega^{-2}$, $\kappa=A^{N^2}$, and~$\iota=\omega^{N^2}$, then for any skein $[K]\in\mathcal{S}^\kappa(S)$ with the vertical framing and projecting to a simple closed curve, we have 
\[
\Tr_T^\omega(F_N([K]))(Z_1,\dots,Z_n)=\Tr_T^\iota([K])(Z_1^N,\dots,Z_n^N).
\]
\end{theorem}

\subsection{Results on Chebyshev polynomials}

In this section, we will introduce a technical tool, the notion of an inverse Chebyshev polynomial. This terminology is not standard and will be used again only in Section~\ref{sec:ProductExpansion}.

\begin{definition}
For a positive integer $k$, an \emph{inverse Chebyshev polynomial} of degree $k$ is a polynomial
\[
\tilde{F}_k(t) = \sum_{i=0}^k c_{k,i} \, t^i  \in \mathbb{Z}[t]
\]
such that
\begin{align*}
({\rm Tr} A)^k = c_{k,0} + \sum_{i=1}^k c_{k,i} \, {\rm Tr}(A^i)
\end{align*}
for all $A\in SL(2,\mathbb{C})$.
\end{definition}

Note that this last equation is different from the equation $({\rm Tr} A)^k = {\rm Tr}( \tilde{F}_k(A) )$ because of the constant term $c_{k,0}$. If we like, we can rewrite it as
\begin{align*}
(\Tr A)^k = {\rm Tr}(\tilde{F}_k(A)) - \tilde{F}_k(0)
\end{align*}
for all $A\in SL(2,\mathbb{C})$ because $\Tr(A^0) = 2$ and $\tilde{F}_k(0)=c_{k,0}$.

\begin{lemma}
For each positive integer $k$, there exists an inverse Chebyshev polynomial of degree $k$, which is a monic polynomial whose coefficients are nonnegative integers.
\end{lemma}

\begin{proof}
We will prove the existence of a monic polynomial $\tilde{F}_k(t)$ of degree $k$ with nonnegative integer coefficients satisfying the desired property by induction. Observe that $\tilde{F}_1(t) =t$ gives the desired result in the case $k=1$. Let $k$ be a positive integer, and suppose that we have a polynomial
\begin{align*}
\tilde{F}_k(t) = \sum_{i=0}^k c_{k,i} \, t^i  \in \mathbb{Z}_{\geq0}[t]
\end{align*}
with $c_{k,k}=1$ which is an inverse Chebyshev polynomial. Observe that for any $A\in SL(2,\mathbb{C})$, we have
\begin{align*}
(\Tr A)^{k+1} &= (\Tr A)^k ({\rm Tr} A)=(\Tr(\tilde{F}_k(A)) - \tilde{F}_k(0)) \cdot (\Tr A) 
= \Tr(\tilde{F}_k(A))\cdot (\Tr A) - c_{k,0} (\Tr A) \\
&= \Tr(\tilde{F}_k(A)\cdot A) + \Tr(\tilde{F}_k(A)\cdot A^{-1})  - c_{k,0} (\Tr A) \\
&= \Tr\left(\sum_{i=0}^k c_{k,i}A^{i+1}\right) + \Tr\left(\sum_{i=0}^k c_{k,i}A^{i-1}\right)  - c_{k,0} (\Tr A) \\
&= \Tr\left(c_{k,k} A^{k+1} + c_{k,k-1} A^k + \left(\sum_{i=1}^{k-1} (c_{k,i-1}+c_{k,i+1}) A^i\right) + c_{k,0} A^{-1} + c_{k,1} A^0 - c_{k,0} A \right) \\
&= \Tr\left( A^{k+1} + c_{k,k-1} A^k + \left(\sum_{i=1}^{k-1} (c_{k,i-1}+c_{k,i+1}) A^i \right)+ c_{k,1} A^0\right)
\end{align*}
where the sums $\sum_{i=1}^{k-1}$ are meant to be zero if $k-1<1$. The last equality comes from
\[
\Tr(A) = \Tr(A^{-1}),
\]
which holds for all $A\in SL(2,\mathbb{C})$ because the the complex eigenvalues of $A$ are always~$\alpha$ and~$\alpha^{-1}$ for some nonzero complex number $\alpha$ and thus ${\rm Tr} (A)$ and ${\rm Tr}(A^{-1})$ both equal $\alpha+\alpha^{-1}$.

Now define a polynomial $\tilde{F}_{k+1}(t) \in \mathbb{Z}[t]$ by 
\begin{align*}
\tilde{F}_{k+1}(t)\coloneqq t^{k+1} + c_{k,k-1} t^k + \left(\sum_{i=1}^{k-1} (c_{k,i-1}+c_{k,i+1}) \, t^i \right) + 2 c_{k,1}
\end{align*}
where the sum $\sum_{i=0}^{k-1}$ is meant to be zero if $k-1<1$.  Then we have $({\rm Tr}A)^{k+1} = {\rm Tr}(\tilde{F}_{k+1}(A)) - \tilde{F}_{k+1}(0)$ for all $A\in SL(2,\mathbb{C})$. Note also that $\tilde{F}_{k+1}(t)$ is monic with degree $k+1$, and that all its coefficients are non-negative integers, because all of $c_{k,k-1},\dots,c_{k,1},c_{k,0}$ are. The existence of $\tilde{F}_k(t)$ for all positive integers $k$ follows by induction.
\end{proof}

\begin{lemma}
\label{lem:inverseChebyshev}
For each positive integer $k$, we have
\[
t^k = F_k(t) + c_{k,k-1} F_{k-1}(t) + \cdots + c_{k,1} F_1(t) + c_{k,0} \in\mathbb{Z}[t]
\]
where $\tilde{F}_k(t) = \sum_{i=0}^k c_{k,i} \, t^i \in \mathbb{Z}_{\geq0}[t]$ is an inverse Chebyshev polynomial of degree~$k$.
\end{lemma}

\begin{proof}
For any $k\ge 1$ and any $A\in SL(2,\mathbb{C})$, we have
\begin{align*}
(\Tr A)^k = c_{k,0} + \sum_{i=1}^k c_{k,i} \Tr(A^i) = c_{k,0} + \sum_{i=1}^k c_{k,i} F_i(\Tr A).
\end{align*}
So if we let $F(t)\coloneqq t^k - c_{k,0} - \sum_{i=1}^k c_{k,i} F_i(t) \in\mathbb{R}[t]$, then $F(\Tr A) = 0$ for all $A\in SL(2,\mathbb{C})$. It follows that $F(t)=0 \in \mathbb{Z}[t]$.
\end{proof}

\subsection{Product expansion}
\label{sec:ProductExpansion}

In this section, we prove a quantum analog of the classical formula for the product $\mathbb{I}(\ell)\mathbb{I}(\ell')$ in Theorem~\ref{thm:classicalproperties}. We begin with some facts about curves on a punctured surface.

Let $\ell$ be a lamination on $S$ consisting of a single peripheral loop of weight~1 that retracts to a puncture~$p$. By applying a homotopy, we may assume that this loop intersects the edges of the ideal triangulation $T$ in the minimal number of points. Orient this loop so that it travels in the counterclockwise direction around the puncture $p$.

\begin{definition}
A triangle $t$ of $T$ will be called \emph{relevant} if one of its vertices is $p$ and \emph{very relevant} if all of its vertices are $p$. Thus a very relevant triangle intersects $\ell$ in three segments as illustrated below.
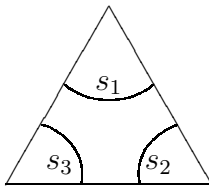
\[
\xy /l1.25pc/:
{\xypolygon3"A"{~:{(-3,0):}}},
(2.1,-1)*{}="i1";
(2.7,0)*{}="i2";
(-0.1,-1)*{}="j1";
(-0.7,0)*{}="j2";
(1.7,1.5)*{}="k1";
(0.2,1.5)*{}="k2";
"i1";"j1" **\crv{(1.5,-0.5) & (0.5,-0.5)};
"j2";"k2" **\crv{(0.5,0.5) & (0.25,1.5)};
"i2";"k1" **\crv{(2.5,0) & (1.5,0.5)};
(1,-1)*{s_1};
(-0.25,1)*{s_2};
(2.25,1)*{s_3};
\endxy
\]
Such a triangle will be called \emph{weird} if we traverse the segments in the cyclic order $s_1$,~$s_2$,~$s_3$ as we travel along the curve $\ell$ in the direction given by the orientation.
\end{definition}

\begin{lemma}
\label{lem:weird}
Among all relevant triangles, there is one that is either not very relevant or very relevant but not weird.
\end{lemma}

\begin{proof}
Suppose for a contradiction that all triangles are very relevant and weird. If we travel once around the curve in the counterclockwise direction, starting from some chosen initial point in the interior of a triangle, then we will meet each edge twice. Let $i_1$ be the first edge that we meet when traveling away from the initial point, and label the other edges of the triangle containing this initial point in its interior as follows:
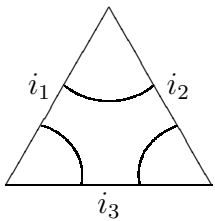
\[
\xy /l1.25pc/:
{\xypolygon3"A"{~:{(-3,0):}}},
(2.1,-1)*{}="i1";
(2.7,0)*{}="i2";
(-0.1,-1)*{}="j1";
(-0.7,0)*{}="j2";
(1.7,1.5)*{}="k1";
(0.2,1.5)*{}="k2";
"i1";"j1" **\crv{(1.5,-0.5) & (0.5,-0.5)};
"j2";"k2" **\crv{(0.5,0.5) & (0.25,1.5)};
"i2";"k1" **\crv{(2.5,0) & (1.5,0.5)};
(2.75,-1)*{i_1};
(-0.75,-1)*{i_2};
(1,2)*{i_3};
\endxy
\]
Then the sequence of edges that we meet as we travel around the curve has the form $i_1,\dots,i_1,i_2,\dots,i_2,i_3,\dots,i_3$. The second edge in this sequence cannot be $i_1$, $i_2$, or~$i_3$, so it must be some other edge $i_4$.
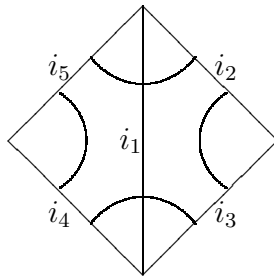
\[
\xy /l1.5pc/:
{\xypolygon4"A"{~:{(2,2):}}},
{"A1"\PATH~={**@{-}}'"A3"},
(2.1,-1.75)*{}="i1";
(2.75,-1)*{}="i2";
(-0.1,-1.75)*{}="j1";
(-0.75,-1)*{}="j2";
(-0.1,1.75)*{}="k1";
(-0.75,1)*{}="k2";
(2.1,1.75)*{}="l1";
(2.75,1)*{}="l2";
"i1";"j1" **\crv{(1.5,-1) & (0.5,-1)};
"j2";"k2" **\crv{(0,-0.5) & (0,0.5)};
"k1";"l1" **\crv{(0.5,1) & (1.5,1)};
"i2";"l2" **\crv{(2,-0.5) & (2,0.5)};
(1.25,0)*{i_1};
(-0.75,-1.55)*{i_2};
(-0.75,1.55)*{i_3};
(2.75,1.55)*{i_4};
(2.75,-1.55)*{i_5};
\endxy
\]
The third side of the triangle having $i_1$ and $i_4$ as sides cannot be $i_2$, $i_3$, or~$i_4$ since each of these has already appeared twice in the sequence. So it must be some other edge $i_5$. Thus the sequence has the form 
\[
i_1,i_4,\dots,i_4,i_5,\dots,i_5,i_1,i_2,\dots,i_2,i_3,\dots,i_3.
\]
The third edge of this sequence cannot be one of the edges $i_1$, $i_2$, $i_3$, $i_4$, or~$i_5$ since each of these has already appeared twice in the sequence. Since there are only finitely many triangles, this process cannot continue forever, and we have a contradiction.
\end{proof}

An edge of $T$ will be called \emph{relevant} if one of its endpoints is $p$. We now describe an algorithm for assigning an orientation to each relevant edge. To do this, we begin by choosing a point on~$\ell$ that lies in the interior of some triangle. We then travel along $\ell$ in the counterclockwise direction around $p$ starting from this point and assign orientations to the edges we meet as follows:
\begin{enumerate}
\item When we meet a relevant edge that has not been assigned an orientation, we give it the ``left'' orientation (pointing into the puncture).

\item When we meet an edge that has already been assigned an orientation, we pass it by, leaving the orientation as it is.
\end{enumerate}
In this way, we assign an orientation to every relevant edge. These orientations are completely determined by the choice of starting point.

\begin{lemma}
\label{lem:nocyclic}
For an appropriate choice of initial point on $\ell$, we can orient the relevant edges of $T$ by the algorithm described above, and there will be no very relevant cyclically oriented triangles.
\end{lemma}

\begin{proof}
Orient $\ell$ in the counterclockwise direction around $p$. By Lemma~\ref{lem:weird}, we know that there exists a triangle which is either not very relevant or very relevant but not weird. Let us choose the initial point to lie in the interior of this triangle. Then we use the algorithm to orient all relevant edges of~$T$.

Suppose there is a very relevant triangle oriented as in the picture below.
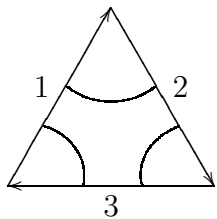
\[
\xy /l1.25pc/:
{\xypolygon3"A"{~:{(-3,0):}}},
{"A2"\ar@{->}"A1"},
{"A3"\ar@{->}"A2"},
{"A1"\ar@{->}"A3"},
(2.1,-1)*{}="i1";
(2.7,0)*{}="i2";
(-0.1,-1)*{}="j1";
(-0.7,0)*{}="j2";
(1.7,1.5)*{}="k1";
(0.2,1.5)*{}="k2";
"i1";"j1" **\crv{(1.5,-0.5) & (0.5,-0.5)};
"j2";"k2" **\crv{(0.5,0.5) & (0.25,1.5)};
"i2";"k1" **\crv{(2.5,0) & (1.5,0.5)};
(2.75,-1)*{1};
(-0.75,-1)*{2};
(1,2)*{3};
\endxy
\]
Consider the algorithm that we used to assign orientations to this triangle. When we first met edge~1, we gave it the left orientation and then followed the curve to edge~2. Since edge~2 has the right orientation, we must have previously assigned its orientation and followed the curve to edge~3. But then we must have previously assigned this orientation and followed the curve to edge~1, a contradiction. On the other hand, suppose there is a very relevant triangle oriented as in the picture below.
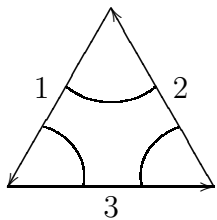
\[
\xy /l1.25pc/:
{\xypolygon3"A"{~:{(-3,0):}}},
{"A1"\ar@{->}"A2"},
{"A2"\ar@{->}"A3"},
{"A3"\ar@{->}"A1"},
(2.1,-1)*{}="i1";
(2.7,0)*{}="i2";
(-0.1,-1)*{}="j1";
(-0.7,0)*{}="j2";
(1.7,1.5)*{}="k1";
(0.2,1.5)*{}="k2";
"i1";"j1" **\crv{(1.5,-0.5) & (0.5,-0.5)};
"j2";"k2" **\crv{(0.5,0.5) & (0.25,1.5)};
"i2";"k1" **\crv{(2.5,0) & (1.5,0.5)};
(2.75,-1)*{1};
(-0.75,-1)*{2};
(1,2)*{3};
\endxy
\]
In this case, one can check that the initial point must lie in the interior of this triangle. Moreover, the triangle must be weird in the sense defined above. This contradicts our choice of the initial point. The lemma follows.
\end{proof}

\begin{lemma}
Let $\ell$ be a lamination consisting of a single peripheral loop of weight~1 on~$S$, and let $K$ be a framed link in $S\times[0,1]$ that projects to $\ell$ with constant elevation and the vertical framing at every point. Then $\Tr_T^\omega([K])$ is a Laurent polynomial in the variables $Z_i$ with coefficients in $\mathbb{Z}_{\geq0}[\omega,\omega^{-1}]$.
\end{lemma}

\begin{proof}
By Lemma~\ref{lem:nocyclic}, we can orient the relevant edges of $T$ so that there are no cyclically oriented triangles. Consider the split ideal triangulation $\widehat{T}$ associated to~$T$. The orientations of the edges of $T$ induce orientations of the corresponding edges of~$\widehat{T}$. If $t$ is a triangle in~$\widehat{T}$, then $t$ is either not very relevant, meaning that it has at most two corners with a segment, or very relevant but not cyclically oriented, so we can choose distinct elevations for the components of $K\cap(t\times[0,1])$ in such a way that there are no crossings over $t$ and the elevations are consistent with the picture conventions following Definition~\ref{def:framedlink}. It follows that we can deform $K$ into good position with respect to $\widehat{T}$ in such a way that the resulting link has no crossings and is consistent with our picture conventions. Then the lemma follows from the definition of the quantum trace $\Tr_T^\omega([K])$.
\end{proof}

\begin{lemma}
\label{lem:traceperipheral}
Let $\ell$ be a lamination consisting of a single peripheral loop of weight~1 on $S$, and let $K$ be a framed link in $S\times[0,1]$ that projects to $\ell$ with constant elevation and the vertical framing at every point. Then 
\[
\Tr_T^\omega([K])=\mathbb{I}^\omega(\ell)+\mathbb{I}^\omega(-\ell).
\]
\end{lemma}

\begin{proof}
Recall that if $\ell$ is a peripheral curve, then the classical functions $\mathbb{I}(\ell)$ and $\mathbb{I}(-\ell)$ are defined as eigenvalues of a matrix $M\in SL_2(\mathbb{C})$, and $\Tr(M)=\mathbb{I}(\ell)+\mathbb{I}(-\ell)$. It follows that $\Tr_T^1([K])=\mathbb{I}^1(\ell)+\mathbb{I}^1(-\ell)$. Since $\Tr_T^\omega([K])$ has coefficients in $\mathbb{Z}_{\geq0}[\omega,\omega^{-1}]$, each of its terms survives in the classical limit, and hence the quantum trace $\Tr_T^\omega([K])$ is the sum of exactly two terms, whose classical limits are $\mathbb{I}^1(\ell)$ and $\mathbb{I}^1(-\ell)$. From Lemma~\ref{lem:BWstar}, we know that $\Tr_T^\omega([K])$ is $*$-invariant, and hence so are its highest and lowest terms. It follows from Lemma~\ref{lem:Weyllimit} that these terms are simply the Weyl orderings of $\mathbb{I}^1(\ell)$ and $\mathbb{I}^1(-\ell)$. The lemma now follows from the definition of $\mathbb{I}^\omega$ on a peripheral loop.
\end{proof}

\begin{lemma}
\label{lem:absorbperipheral}
If $\ell$ is a peripheral loop and $\ell'$ is any lamination, then 
\[
\mathbb{I}^\omega(\ell)\mathbb{I}^\omega(\ell')=\mathbb{I}^\omega(\ell+\ell').
\]
\end{lemma}

\begin{proof}
Since $\ell$ is peripheral, we can move it into a small region of a puncture disjoint from all curves of $\ell'$. Then the union of the curves of $\ell$ with the curves of $\ell'$ represents the lamination $\ell+\ell'$. By definition, $\mathbb{I}^\omega(\ell)\mathbb{I}^\omega(\ell')=\mathbb{I}^\omega(\ell+\ell')$.
\end{proof}

\begin{lemma}
\label{lem:easyproductexpansion}
Let $\ell$ and $\ell'$ be two laminations on~$S$, each represented by a single nonperipheral curve of weight~1, and let $k$ and~$k'$ be positive integers. Then 
\[
\mathbb{I}^\omega(\ell)^{k}\,\mathbb{I}^\omega(\ell')^{k'}=\sum_{\ell''\in\mathcal{A}_L(S,\mathbb{Z})}c^\omega(\ell,\ell',k,k';\ell'')\mathbb{I}^\omega(\ell'')
\]
where $c^\omega(\ell,\ell',k,k';\ell'')\in\mathbb{Z}[\omega^2,\omega^{-2}]$ and only finitely many terms are nonzero.
\end{lemma}

\begin{proof}
Let $K$ and $K'$ be framed links with constant elevation and vertical framing that project to the curves $\ell$ and $\ell'$, respectively. Then each of the skeins $[K]$ and $[K']$ has a projected diagram with no crossings and no kinks. Therefore, so do the superimposed skeins $[K]^{k}$ and $[K']^{k'}$. The projected diagram for $[K]^{k}$ is a disjoint union of parallel copies of that of $[K]$, and likewise for~$[K']^{k'}$. Note that 
\begin{align*}
\mathbb{I}^\omega(\ell)^{k}\,\mathbb{I}^\omega(\ell')^{k'} &= \Tr_T^\omega([K])^{k} \Tr_T^\omega([K'])^{k'} \\
&= \Tr_T^\omega([K]^{k}[K']^{k'}) \\
&= \Tr_T^\omega([K^k{K'}^{k'}]).
\end{align*}
By applying the skein relations, we can write
\[
[K^k{K'}^{k'}]=\sum_{i=1}^N\omega^{2p_i}[K_i].
\]
where $N$ and $p_i$ are integers and each $K_i$ is a link equipped with the vertical framing at every point whose projection to $S$ has no crossings or kinks. Then 
\[
\Tr_T^\omega([K^k{K'}^{k'}])=\sum_{i=1}^N\omega^{2p_i}\Tr_T^\omega([K_i]).
\]

We now claim that each $\Tr_T^\omega([K_i])$ can be written as a finite linear combination of expressions $\mathbb{I}^\omega(\ell)$ with coefficients in $\mathbb{Z}[\omega^2,\omega^{-2}]$. Indeed, we may assume that $K_i$ has no contractible curves since any contractible curve equals $-(\omega^4+\omega^{-4})$ in the skein algebra. Then each $K_i$ projects to a collection of nonintersecting and noncontractible curves on~$S$. Write $C_{i,1},\dots,C_{i,M_i}$ for the components of this projection. Let $\tilde{C}_{i,1},\dots,\tilde{C}_{i,M_i}$ be framed links obtained by embedding these curves at constant elevations with vertical framing. Then the skeins $[\tilde{C}_{i,1}],\dots,[\tilde{C}_{i,M_i}]$ commute, and we have 
\[
[K_i]=[\tilde{C}_{i,1}]\dots[\tilde{C}_{i,M_i}].
\]
Since $\Tr_T^\omega$ is an algebra homomorphism, we have 
\[
\Tr_T^\omega([K_i])=\Tr_T^\omega([\tilde{C}_{i,1}])\dots\Tr_T^\omega([\tilde{C}_{i,M_i}]).
\]

By Lemma~\ref{lem:traceperipheral} and the definition of $\mathbb{I}^\omega$, the right hand side of this expression can be written as a finite sum of terms of the form 
\[
\mathbb{I}^\omega(\ell_{t_1})\mathbb{I}^\omega(\ell_{t_2})\dots\mathbb{I}^\omega(\ell_{t_{M_i}})
\]
where, for each $j$, \ $\ell_j$ denotes a lamination represented by a single loop of weight~1 or~$-1$, whose homotopy class coincides with the homotopy class of $C_{i,j}$. So, for any distinct indices $j$ and~$j'$, the laminations $\ell_{t_j}$ and $\ell_{t_{j'}}$ are either disjoint or represented by the same curve. Thus, in particular, all of the factors in this expression commute, and thus we can rewrite the expression as 
\[
\mathbb{I}^\omega(\ell_{s_1})^{R_1}\mathbb{I}^\omega(\ell_{s_2})^{R_2}\dots \mathbb{I}^\omega(\ell_{s_{m_i}})^{R_{m_i}}
\]
where $m_i$ is a positive integer, $R_1,\dots,R_{m_i}$ are positive integers, $\ell_{s_1},\dots,\ell_{s_{m_i}}$ are mutually disjoint and mutually nonhomotopic laminations, each represented by a single loop on~$S$.

Suppose $\ell_{s_j}$ is represented by a nonperipheral curve. Then by Lemma~\ref{lem:inverseChebyshev}, we have 
\[
\mathbb{I}^\omega(\ell_{s_j})^{R_j}=\sum_{b=0}^{R_j}c_{R_j,b}\, \mathbb{I}^\omega(b\ell_{s_j})
\]
where $c_{R_j,b}$ are nonnegative integers. Therefore, we can write the previous displayed expression as a finite $\mathbb{Z}_{\geq0}$-linear combination of terms of the form 
\[
\mathbb{I}^\omega(b_1\ell_{s_1})\mathbb{I}^\omega(b_2\ell_{s_2})\dots\mathbb{I}^\omega(b_{m_i}\ell_{m_i})
\]
for nonnegative integers $b_1,\dots,b_{m_i}$, and this equals 
\[
\mathbb{I}^\omega(b_1\ell_{s_1}+b_2\ell_{s_s}+\dots+b_{m_i}\ell_{s_{m_i}}).
\]
This completes the proof.
\end{proof}

\begin{lemma}
\label{lem:multipleproductexpansion}
Let $\ell$ and $\ell'$ be two laminations on~$S$, each represented by a single nonperipheral curve of weight~1, and let $k$ and $k'$ be positive integers. Then 
\[
\mathbb{I}^\omega(k\ell)\mathbb{I}^\omega(k'\ell')=\sum_{\ell''\in\mathcal{A}_L(S,\mathbb{Z})}c^\omega(k\ell,k'\ell';\ell'')\mathbb{I}^\omega(\ell'')
\]
where $c^\omega(k\ell,k'\ell';\ell'')\in\mathbb{Z}[\omega^2,\omega^{-2}]$ and only finitely many terms are nonzero.
\end{lemma}

\begin{proof}
For convenience, let us give names to the coefficients of the Chebyshev polynomials:
\[
F_k(t)=\sum_{b=0}^kd_{k,b}t^b\in\mathbb{Z}[t].
\]
Observe that 
\begin{align*}
\mathbb{I}^\omega(k\ell)\mathbb{I}^\omega(k'\ell') &= F_k(\mathbb{I}^\omega(\ell))F_{k'}(\mathbb{I}^\omega(\ell')) \\
&= \left(\sum_{b=0}^kd_{k,b}\mathbb{I}^\omega(\ell)^b\right) \left(\sum_{b'=0}^{k'}d_{k',b'}\mathbb{I}^\omega(\ell')^{b'}\right) \\
&= \sum_{b=0}^k\sum_{b'=0}^{k'}d_{k,b}d_{k',b'}\mathbb{I}^\omega(\ell)^{b} \mathbb{I}^\omega(\ell')^{b'}.
\end{align*}
It now suffices to show that we can express $\mathbb{I}^\omega(\ell)^{b} \mathbb{I}^\omega(\ell')^{b'}$ as a linear combination of expressions $\mathbb{I}^\omega(\ell'')$ with coefficients in~$\mathbb{Z}[\omega^2,\omega^{-2}]$. When $b=0$ or~$b'=0$, we can do this using the inverse Chebyshev polynomials. When both $b$ and $b'$ are positive integers, the claim follows from~Lemma~\ref{lem:easyproductexpansion}.
\end{proof}

\begin{proposition}
\label{prop:productexpansion}
If $\ell$ and $\ell'$ are any laminations on~$S$, then 
\[
\mathbb{I}^\omega(\ell)\mathbb{I}^\omega(\ell')=\sum_{\ell''\in\mathcal{A}_L(S,\mathbb{Z})}c^\omega(\ell,\ell';\ell'')\mathbb{I}^\omega(\ell'')
\]
where $c^\omega(\ell,\ell';\ell'')\in\mathbb{Z}[\omega^2,\omega^{-2}]$ and only finitely many terms are nonzero.
\end{proposition}

\begin{proof}
Write 
\begin{align*}
\ell &= \sum_ik_i\ell_i, \\
\ell' &= \sum_ik_i'\ell_i'
\end{align*}
where $\ell_i$ (respectively, $\ell_i'$) are the curves of $\ell$ (respectively, $\ell'$) with each homotopy class of curves appearing at most once in the sum and $k_i$,~$k_i'\in\mathbb{Z}$. Then 
\begin{align*}
\mathbb{I}^\omega(\ell) &= \prod_i\mathbb{I}^\omega(k_i\ell_i), \\
\mathbb{I}^\omega(\ell') &= \prod_i\mathbb{I}^\omega(k_i'\ell_i'),
\end{align*}
and the proposition follows from repeated applications of Lemmas~\ref{lem:multipleproductexpansion} and~\ref{lem:absorbperipheral}.
\end{proof}

\section{Main result}
\label{sec:MainResult}

\subsection{The notion of parity}

We have now proved most of the statements that we need to establish our main result, but it remains to show that the map $\mathbb{I}^\omega:\mathcal{A}_L(S,\mathbb{Z})\rightarrow\mathcal{Z}_T^\omega$ gives rise to a map $\widehat{\mathbb{I}}^q:\mathcal{A}_{SL_2,S}(\mathbb{Z}^t)\rightarrow\mathcal{X}_T^q$. The proof of this will require a digression.

\begin{definition}
An ordered collection $\mathcal{X}$ of elements $x_1,x_2,\dots,x_n$ in an algebra over~$\mathbb{C}$ is said to be \emph{$\omega^2$-commuting} if there is a skew-symmetric integer matrix $\varepsilon_{ij}$ such that
\[
x_i x_j = \omega^{2\varepsilon_{ij}} x_j x_i, \qquad \mbox{for all $i,j \in \{1,\dots,n\}$. }
\]
In this case, a single-term Laurent polynomial in $x_1,\dots,x_n$ with coefficient of the form $\omega^N$ for some integer $N$ is called a \emph{Laurent $\omega$-monomial} in $x_1,\dots,x_n$, or in $\mathcal{X}$. When a Laurent $\omega$-monomial in $x_1,\dots,x_n$ is written as
\[
\omega^N x_1^{p_1} x_2^{p_2} \cdots x_n^{p_n}
\]
for some integers $N,p_1,p_2,\dots,p_n$, we say that it is written in \emph{standard form}.
\end{definition}

Note that any Laurent $\omega$-monomial in $x_1,\dots,x_n$ can be transformed into a unique standard form using the commutation relations.

\begin{definition}
\label{def:parity_compatibility}
Let $\mathcal{X}$ be an ordered collection $x_1,\dots,x_n$ of $\omega^2$-commuting variables. Two Laurent $\omega$-monomials in $\mathcal{X}$, whose standard forms are $\omega^N x_1^{p_1} \cdots x_n^{p_n}$ and $\omega^{N'} x_1^{p_1'} \cdots x_n^{p_n'}$, are called \emph{parity compatible} or \emph{$\mathcal{X}$-parity compatible}, when:
\[
p_i \equiv p_i' \quad({\rm mod}~2) \quad\mbox{for all $i=1,\dots,n$}, \qquad \mbox{and} \quad N\equiv N' \quad({\rm mod}~4).
\]
Parity compatibility defines an equivalence relation, indicated by $\sim$, on the set of Laurent $\omega$-monomials in $\mathcal{X}$.
\end{definition}

\begin{lemma}
\label{lem:squares_are_absorbed}
For any Laurent $\omega$-monomial in $\mathcal{X}$, say $\omega^N x_{i_1} \cdots x_{i_m}$, and for any $i \in \{1,\dots,n\}$, one has
\[
x_i^2 \, (\omega^N x_{i_1} \cdots x_{i_m}) \sim (\omega^N x_{i_1} \cdots x_{i_m}) \sim (\omega^N x_{i_1} \cdots x_{i_m}) \, x_i^2.
\]
\end{lemma}

\begin{proof}
Write $\omega^N x_{i_1} \cdots x_{i_m}$ in its standard form, say $\omega^{N'} x_1^{p_1} \cdots x_n^{p_n}$.  Note that $x_i^2 x_j = \omega^{4\varepsilon_{ij}} x_j x_i^2$, for each $j$. Thus
\begin{align*}
x_i^2 \, (\omega^{N'} x_1^{p_1} \cdots x_n^{p_n}) & = \omega^{N' + 4N''} \, x_1^{p_1} \cdots x_i^{p_i+2} \cdots x_n^{p_n}, \\
 (\omega^{N'} x_1^{p_1} \cdots x_n^{p_n})\, x_i^2 & = \omega^{N' + 4N'''} \, x_1^{p_1} \cdots x_i^{p_i+2} \cdots x_n^{p_n}
\end{align*}
for some integers $N'',N'''$. Definition~\ref{def:parity_compatibility} says that both of these are parity compatible with $(\omega^{N'} x_1^{p_1} \cdots x_n^{p_n})$, hence with $\omega^N x_{i_1} \cdots x_{i_m}$.
\end{proof}

\begin{lemma}
\label{lem:multiplication_of_a_common_factor_preserves_parity_compatibility}
Let $f_1$, $f_2$, $f_3$, $f_4$ be Laurent $\omega$-monomials in $\mathcal{X}$. Then
\[
f_1 \sim f_2 ~ \Rightarrow ~ f_3 \, f_1 \, f_4 \sim f_3 \, f_2 \, f_4.
\]
\end{lemma}

\begin{proof}
Write $f_1$ and $f_2$ in standard forms:
\[
f_1 =\omega^N x_1^{p_1} \cdots x_n^{p_n}, \qquad
f_2 =\omega^{N'} x_1^{p_1'} \cdots x_n^{p_n'},
\]
so we have $p_i \equiv p_i'$ (mod $2$) for all $i=1,\dots,n$ and $N\equiv N'$ (mod $4$). Let us prove the lemma only for the case $f_3=1$ and $f_4 = x_j$ for some $j$. Note that $x_i^{p_i} x_j = \omega^{2\varepsilon_{ij} \, p_i} x_j x_i^{p_i}$, so
\begin{align*}
f_3 \, f_1 \, f_4 & = \omega^N x_1^{p_1} \cdots x_n^{p_n} x_j
= \omega^{N + 2\sum_{i>j} \varepsilon_{ij} p_i } \, x_1^{p_1} \cdots x_j^{p_j+1} \cdots x_n^{p_n} \\
f_3 \, f_2 \, f_4 & = \omega^{N'} x_1^{p_1'} \cdots x_n^{p_n'} x_j = \omega^{N' + 2\sum_{i>j} \varepsilon_{ij} p_i' } \, x_1^{p_1'} \cdots x_j^{p_j'+1} \cdots x_n^{p_n'}.
\end{align*}
Note that
\[
\left(N + 2\sum_{i>j}\varepsilon_{ij}p_i\right) -\left(N' + 2\sum_{i>j}\varepsilon_{ij}p_i'\right)
= (N-N') + \sum_{i>j} \varepsilon_{ij} \, 2(p_i - p_i') \equiv 0 \quad({\rm mod}~4)
\]
because $N-N' \equiv 0$ (mod $4$) and $p_i-p_i' \equiv 0$ (mod $2$). So the lemma holds in the case when $f_3 = 1$ and $f_4 = x_j$. A similar proof works for the case when $f_3=x_j$ and $f_4=1$. For general $f_3$ and $f_4$, one can apply these simple cases repeatedly.
\end{proof}

\begin{corollary}
\label{cor:multiplicativity_of_parity_compatibility}
Let $f_1$, $f_2$, $f_3$, $f_4$ be Laurent $\omega$-monomials in $\mathcal{X}$. Then
\[
f_1 \sim f_2 \mbox{ and } f_3 \sim f_4 ~ \Rightarrow ~ f_1f_3 \sim f_2f_4.
\]
\end{corollary}

\begin{proof}
One notes that $f_1 \sim f_2$ implies $f_1 f_3 \sim f_2 f_3$ and $f_3 \sim f_4$ implies $f_2f_3 \sim f_2 f_4$, from Lemma~\ref{lem:multiplication_of_a_common_factor_preserves_parity_compatibility}. The result then follows from the transitivity of $\sim$.
\end{proof}

\begin{definition}
\label{def:parity_of_Laurent_omega-monomials}
Let $Q\in \mathbb{Z}/4\mathbb{Z}$ and ${\bf q} = (q_1,\dots,q_n) \in (\mathbb{Z}/2\mathbb{Z})^n$. A Laurent $\omega$-monomial in~$\mathcal{X}$, say $\omega^N x_1^{p_1} \cdots x_n^{p_n}$ in its standard form, is said to be of \emph{parity $(Q, {\bf q})$} if $p_i \equiv q_i$ (mod $2$) for all $i=1,\dots,n$ and $N \equiv Q$ (mod $4$).
\end{definition}

We note the following simple facts about parity:

\begin{lemma}
\label{lem:uniqueness_of_parity}
A Laurent $\omega$-monomial in $\mathcal{X}$ is of parity $(Q,{\bf q})$ for unique $Q \in \mathbb{Z}/4\mathbb{Z}$ and ${\bf q} \in (\mathbb{Z}/2\mathbb{Z})^n$. \qed
\end{lemma}

\begin{lemma}
\label{lem:parity_compatibility_means_same_parity}
Two Laurent $\omega$-monomials in $\mathcal{X}$ are parity compatible if and only if both of them are of parity $(Q,{\bf q})$ for some $Q\in \mathbb{Z}/4\mathbb{Z}$ and ${\bf q} \in (\mathbb{Z}/2\mathbb{Z})^n$. \qed
\end{lemma}

Later we will need to consider sums of Laurent $\omega$-monomials in $\mathcal{X}$.

\begin{definition}
Let $\mathcal{X}$ be a collection of variables $x_1,\dots,x_n$. Then define
\[
\mathbf{Z}^\omega \{\mathcal{X}\} \coloneqq  \mathbf{Z}^\omega\langle x_1,x_1^{-1}, \dots, x_n, x_n^{-1}\rangle / ( x_j x_j^{-1} = x_j^{-1} x_j=1, ~ x_j x_k = \omega^{2\varepsilon_{jk}} x_k x_j )
\]
where $\mathbf{Z}^\omega \coloneqq \mathbb{Z}[\omega,\omega^{-1}]$.
\end{definition}

Equivalently, $\mathbf{Z}^\omega \{\mathcal{X}\}$ is the algebra generated by Laurent $\omega$-monomials in $\mathcal{X}$, with multiplication given by the distributive law. It can be decomposed into subgroups according to parity.

\begin{definition}
For $Q \in \mathbb{Z}/4\mathbb{Z}$ and ${\bf q}\in(\mathbb{Z}/2\mathbb{Z})^n$, let $(\mathbf{Z}^\omega\{\mathcal{X}\})_{(Q,{\bf q})}$ be the additive subgroup of $\mathbf{Z}^\omega\{\mathcal{X}\}$ generated by those Laurent $\omega$-monomials in~$\mathcal{X}$ that have parity $(Q,{\bf q})$.
\end{definition}

In view of Lemma~\ref{lem:parity_compatibility_means_same_parity}, we say that two nonzero elements of $\mathbf{Z}^\omega \{\mathcal{X}\}$ are \emph{parity compatible} if they lie in some common $(\mathbf{Z}^\omega\{\mathcal{X}\})_{(Q,{\bf q})}$. In this case, they are said to be of \emph{parity} $(Q,\mathbf{q})$.

\begin{lemma}
One has
\[
\mathbf{Z}^\omega \{\mathcal{X}\} = \bigoplus_{(Q,{\bf q}) \in \mathbb{Z}/4\mathbb{Z} \times (\mathbb{Z}/2\mathbb{Z})^n} (\mathbf{Z}^\omega\{\mathcal{X}\})_{(Q,{\bf q})}. \qed
\]
\end{lemma}

\begin{lemma}
\label{lem:multiplication_of_parity_subspaces}
The parity is ``multiplicative'' in the following sense: For any two elements $(Q,{\bf q})$ and $(Q', {\bf q}')$ of $\mathbb{Z}/4\mathbb{Z} \times (\mathbb{Z}/2\mathbb{Z})^n$, there exists a unique third element $(Q'',{\bf q}'')$ such that
\[
(\mathbf{Z}^\omega \{\mathcal{X}\} )_{(Q,{\bf q})} \cdot (\mathbf{Z}^\omega\{\mathcal{X}\} )_{(Q',{\bf q}')} \subseteq (\mathbf{Z}^\omega\{\mathcal{X}\})_{(Q'',{\bf q}'')}.
\]
\end{lemma}

\begin{proof}
Let $f\in (\mathbf{Z}^\omega\{\mathcal{X}\})_{(Q,{\bf q})}$ and  $f'\in (\mathbf{Z}^\omega\{\mathcal{X}\})_{(Q',{\bf q}')}$ be Laurent $\omega$-monomials. Then their product $ff'$ is again a Laurent $\omega$-monomial in~$\mathcal{X}$, hence is of parity $(Q'',{\bf q}'')$ for a unique $(Q'',{\bf q}'')$, by Lemma~\ref{lem:uniqueness_of_parity}. Pick other Laurent $\omega$-monomials in $\mathcal{X}$, say $g\in (\mathbf{Z}^\omega(\{\mathcal{X}\})_{(Q,{\bf q})}$ and $g' \in (\mathbf{Z}^\omega(\{\mathcal{X}\})_{(Q',{\bf q}')}$.  Since $f\sim g$ and $f' \sim g'$, we get $f f' \sim g g'$ from Corollary~\ref{cor:multiplicativity_of_parity_compatibility}, and hence Lemma~\ref{lem:parity_compatibility_means_same_parity} tells us that $g g'$ is also of parity $(Q'',{\bf q}'')$.
\end{proof}

\begin{lemma}
\label{lem:multiplication_by_zero_parity}
In Lemma~\ref{lem:multiplication_of_parity_subspaces}, if $\mathbf{q}'=\mathbf{0}=(0,\ldots,0) \in (\mathbb{Z}/2\mathbb{Z})^n$, then $(Q'', \mathbf{q}'') = (Q+Q', {\bf q})$. If $\mathbf{q}=\mathbf{0}$, then $(Q'', \mathbf{q}'') = (Q+Q', \mathbf{q}')$.
\end{lemma}

\begin{proof}
Note that $\omega^N \in ({\bf Z}^\omega\{\mathcal{X}\})_{(Q',{\bf 0})}$ where $N\equiv Q'$ (mod~$4$). The product of any Laurent $\omega$-monomial in $\mathcal{X}$ belonging to $(\mathbf{Z}^\omega\{\mathcal{X}\})_{(Q,{\bf q})}$ and $\omega^N \in ({\bf Z}^\omega\{\mathcal{X}\})_{(Q',{\bf 0})}$, regardless of the product order, belongs to $({\bf Z}^\omega\{\mathcal{X}\})_{(Q+Q',{\bf q})}$. The claim now follows from Lemma~\ref{lem:multiplication_of_parity_subspaces}.
\end{proof}

\subsection{Parity compatibility of terms}

We will now apply the general theory of the previous section to Laurent polynomials of the form~$\mathbb{I}^\omega(\ell)$ for $\ell$ a lamination. In this section, when we speak of $\mathscr{Z}$-parity compatible monomials, we are taking $\mathscr{Z}$ to be the set of generators $Z_1,\dots,Z_n$ of the Chekhov-Fock algebra associated to an ideal triangulation.

Fix an ideal triangulation $T$, and let $\widehat{T}$ be the associated split ideal triangulation. Let $t_1,\dots,t_m$ and $B_1,\dots,B_n$ be the triangles and biangles of $\widehat{T}$, respectively. If $K\subseteq S\times[0,1]$ is a link with the vertical framing having a projected diagram with no crossings and no kinks, then we will put it in good position with respect to~$\widehat{T}$ as follows. First, deform $K$ by an isotopy so that its projection intersects each edge of $\widehat{T}$ in the minimal number of points. Choose a point where the projection intersects an edge of a biangle~$B_{i_1}$. This point lies on an edge of some triangle $t_{j_1}$, and we can deform~$K$ so that it travels over this triangle at constant elevation before coming to another biangle~$B_{i_2}$. We then deform the link so that its elevation decreases while passing over $B_{i_2}$ until it reaches another triangle $t_{j_2}$. We deform~$K$ so that it passes over this triangle at constant elevation before coming to another biangle~$B_{i_3}$, where its elevation decreases. Continue this process until the curve returns to $B_{i_1}$ for the last time. Since the curve is closed, its elevation must increase as it crosses $B_{i_1}$.

We now orient the edges of each biangle so that the two sides point in the same direction. This allows us to draw the projection to $S$ using the picture conventions of Bonahon and Wong, reviewed in Section~\ref{sec:Preliminaries}. For every biangle other than $B_{i_1}$, the projection has no intersections since the link is always decreasing in elevation. On the other hand, the projection to $B_{i_1}$ will have the following form.
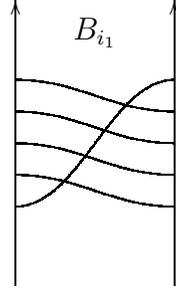
\[
\xy 0;/r.50pc/: 
(5,10)*{B_{i_1}}; 
(0,12)*{}="B"; 
(10,12)*{}="C";
(0,-6)*{}="B1"; 
(10,-6)*{}="C1";
(0,7)*{}="X1"; 
(10,5)*{}="Y1";
(0,5)*{}="X2"; 
(10,3)*{}="Y2";
(0,3)*{}="X3"; 
(10,1)*{}="Y3";
(0,1)*{}="X4"; 
(10,-1)*{}="Y4";
(0,-1)*{}="X0"; 
(10,7)*{}="Y0";
{"B1"\ar@{->}"B"},
{"C1"\ar@{->}"C"},
"X1";"Y1" **\crv{(4,7) & (6,5)};
"X2";"Y2" **\crv{(4,5) & (6,3)};
"X3";"Y3" **\crv{(4,3) & (6,1)};
"X4";"Y4" **\crv{(4,1) & (6,-1)};
"X0";"Y0" **\crv{(4,-1) & (6,7)};
\endxy
\]

Let $K_j=K\cap(t_j\times[0,1])$ and $L_i=K\cap(B_i\times[0,1])$ be the links associated with the triangles and biangles, respectively. Using the skein relations, we can write 
\[
[L_{i_1}]=\sum_{\alpha=1}^N\omega^{2p_\alpha}[L_{i_1,\alpha}]
\]
where $N$ and $p_\alpha$ are integers and each $L_{i_1,\alpha}$ is a link with the vertical framing and no crossings. Suppose we are given a state $\tau_{i_1}:\partial L_{i_1}\rightarrow\{+,-\}$. Then for each $\alpha$, we get a number $\Tr_{B_{i_1}}^\omega([L_{i_1,\alpha},\tau_{i_1}])$ by Definition~\ref{def:computebiangle}. It is expressed as a Laurent monomial in $\omega$.

\begin{lemma}
\label{lem:coefficientcompatibility}
Let $[L_{i_1,\alpha_1}]$ and $[L_{i_1,\alpha_2}]$ be two different resolutions of the skein $[L_{i_1}]$. If $\Tr_{B_{i_1}}^\omega([L_{i_1,\alpha_1},\tau_{i_1}])$ and $\Tr_{B_{i_1}}^\omega([L_{i_1,\alpha_2},\tau_{i_1}])$ are both nonzero, then they are $\mathscr{Z}$-parity compatible.
\end{lemma}

\begin{proof}
Note that any resolution of $[L_{i_1}]$ has the same number of curves ending on either side of the biangle $B_{i_1}$. It follows that there must be equal numbers of components of the types illustrated in the middle and rightmost pictures immediately preceding Definition~\ref{def:computebiangle}. (Here we are ignoring the values of the associated states.)

Let $[L_{i_1,\alpha_1}]$ and $[L_{i_1,\alpha_2}]$ be two different resolutions of the skein $[L_{i_1}]$. For any resolution $[L_{i_1,\alpha}]$ of this skein $[L_{i_1}]$, we have numbers $b_{\sigma_2}^{\sigma_1}$ and $c_{\sigma_2}^{\sigma_1}$ as in Definition~\ref{def:computebiangle}. By the above remarks, we know that when we pass from $[L_{i_1,\alpha_1}]$ to $[L_{i_1,\alpha_2}]$, the numbers  $b_-^++b_+^-$ and $c_-^++c_+^-$ change by the same amount. This fact, combined with the formula in Definition~\ref{def:computebiangle}, implies that $\Tr_{B_{i_1}}^\omega([L_{i_1,\alpha_1},\tau_{i_1}])$ and $\Tr_{B_{i_1}}^\omega([L_{i_1,\alpha_2},\tau_{i_1}])$ are $\mathscr{Z}$-parity compatible, provided both are nonzero.
\end{proof}

Recall that the quantum trace $\Tr_T^\omega([K])$ is defined as a sum 
\[
\Tr_T^\omega([K])=\sum_{\text{compatible }\sigma_j, \tau_i}\prod_{i=1}^n\Tr_{B_i}^\omega([L_i,\tau_i])\bigotimes_{j=1}^m\Tr_{t_j}^\omega([K_j,\sigma_j])
\]
over mutually compatible states $\sigma_j:\partial K_j\rightarrow\{+,-\}$ and $\tau_i:\partial L_i\rightarrow\{+,-\}$. The factor $\Tr_{t_j}^\omega([K_j,\sigma_j])$ in this expression is given by 
\[
\Tr_{t_j}^\omega([K_j,\sigma_j])=\Tr_{t_j}^\omega([k_1,\sigma_j])\Tr_{t_j}^\omega([k_2,\sigma_j])\dots \Tr_{t_j}^\omega([k_l,\sigma_j])
\]
where $k_1,\dots,k_l$ are the components of $K_j$ and $\Tr_{t_j}^\omega([k_i,\sigma_j])$ is either~0 or a Weyl ordered product of two generators of a triangle algebra, depending on the state~$\sigma_j$. For $i=1,\dots,n$, let us choose a resolution $[L_{i,\alpha}]$ of the skein $[L_i]$ over the biangle $B_i$ using the skein relations. Since we have deformed $K$ into a good position where the only crossings over a biangle occur over $B_{i_1}$, this amounts to a choice of resolution $[L_{i_1,\alpha}]$ of $[L_{i_1}]$. Given a set $\mathbf{s}=\{\sigma_j,\tau_i\}$ of mutually compatible states, we define 
\[
\Tr_T^\omega(\mathbf{s},\alpha)\coloneqq\prod_{i=1}^n\Tr_{B_i}^\omega([L_{i,\alpha},\tau_i])\bigotimes_{j=1}^m\Tr_{t_j}^\omega([K_j,\sigma_j]).
\]

We will also consider a slight modification of $\Tr_T^\omega(\mathbf{s},\alpha)$. Precisely, we replace each factor $\Tr_{t_j}^\omega([k_i,\sigma_j])$ in the above definition by the expression 
\[
(\Tr_{t_j}^\omega)'([k_i,\sigma_j])\coloneqq[Z_{j1}^{\sigma_1}Z_{j2}^{\sigma_2}]
\]
from part~\ref{part:trianglecase} of Theorem~\ref{thm:BWmain} and denote their product by $(\Tr_{t_j}^\omega)'([K_j,\sigma_j])$. We then define 
\[
(\Tr_T^\omega)'(\mathbf{s},\alpha)\coloneqq\prod_{i=1}^n\Tr_{B_i}^\omega([L_{i,\alpha},\tau_i])\bigotimes_{j=1}^m(\Tr_{t_j}^\omega)'([K_j,\sigma_j]).
\]
This modified expression is equal to $\Tr_T^\omega(\mathbf{s},\alpha)$ whenever the latter is nonzero.

\begin{lemma}
\label{lem:statecompatibility}
Let $\mathbf{s}_1=\{\sigma_{1j},\tau_{1i}\}$ and $\mathbf{s}_2=\{\sigma_{2j},\tau_{2i}\}$ be two collections of mutually compatible states for the links $K_j$ and $L_i$. If $(\Tr_T^\omega)'(\mathbf{s}_1,\alpha_1)$ and $(\Tr_T^\omega)'(\mathbf{s}_2,\alpha_2)$ are both nonzero, then they are $\mathscr{Z}$-parity compatible.
\end{lemma}

\begin{proof}
Suppose that $\mathbf{s}_2$ is obtained from $\mathbf{s}_1$ by simultaneously changing the common signs associated to the points on either side of a strand that crosses a biangle $B_e$. Suppose~$t_j$ and~$t_J$ are distinct triangles on either side of $B_e$, and let $Z_{je}$ and $Z_{Je}$ be the generators of the associated triangle algebras corresponding to these edges. Then it is straightforward to verify that 
\[
(\Tr_{t_j}^\omega)'([K_j,\sigma_{1j}])=\omega^{4N_j\pm2}Z_{je}^{\pm2}(\Tr_{t_j}^\omega)'([K_j,\sigma_{2j}])
\]
and 
\[
(\Tr_{t_J}^\omega)'([K_J,\sigma_{1J}])=\omega^{4N_J\pm2}Z_{Je}^{\pm2}(\Tr_{t_J}^\omega)'([K_J,\sigma_{2J}])
\]
for integers $N_j$ and $N_J$ where the signs are either all $+$ or all $-$. These are the only factors of $(\Tr_T^\omega)'(\mathbf{s}_1,\alpha)$ that differ from the corresponding factors of $(\Tr_T^\omega)'(\mathbf{s}_2,\alpha)$. It follows that 
\[
(\Tr_T^\omega)'(\mathbf{s}_1,\alpha) \sim (\Tr_T^\omega)'(\mathbf{s}_2,\alpha)
\]
provided both are nonzero. A similar argument shows that $(\Tr_T^\omega)'(\mathbf{s}_1,\alpha)\sim (\Tr_T^\omega)'(\mathbf{s}_2,\alpha)$ in the case where the biangle $B_e$ is associated with the internal edge of a self-folded triangle, provided both are nonzero.

Next, suppose that $\mathbf{s}_2$ is obtained from $\mathbf{s}_1$ by simultaneously changing a~$+$ sign and a~$-$ sign associated to points on the same side of a biangle $B_e$, connected by a strand in~$B_e$. Let $t_j$ be the triangle adjacent to this edge. It is straightforward to verify that 
\[
(\Tr_{t_j}^\omega)'([K_j,\sigma_{1j}])=\omega^{4N_j}(\Tr_{t_j}^\omega)'([K_j,\sigma_{2j}])
\]
for some integer $N_j$. In this case, the factor $\Tr_{B_e}^\omega([L_{e,\alpha},\tau_e])$ also changes by a factor of $\omega^{\pm4}$. It follows that 
\[
(\Tr_T^\omega)'(\mathbf{s}_1,\alpha) \sim (\Tr_T^\omega)'(\mathbf{s}_2,\alpha)
\]
provided both expressions are nonzero.

The result now follows by applying finitely many elementary moves of the above types together with Lemma~\ref{lem:coefficientcompatibility}.
\end{proof}

Note that, by definition, $\Tr_T^\omega([K])$ is given by 
\[
\Tr_T^\omega([K])=\sum_{\mathbf{s},\alpha}\omega^{2p_\alpha}\Tr_T^\omega(\mathbf{s},\alpha)
\]
where the sum is over all collections $\mathbf{s}$ of compatible states and all $\alpha$. Hence Lemma~\ref{lem:statecompatibility} implies the following.

\begin{lemma}
\label{lem:parity_compatibility_of_each_term}
If $K\subseteq S\times[0,1]$ is a link with the vertical framing and no crossings, then all terms of $\Tr_T^\omega([K])$ are $\mathscr{Z}$-parity compatible.\qed
\end{lemma}

Note that if all terms of an element $f\in\mathbf{Z}^\omega\{\mathscr{Z}\}$ are parity compatible, then in particular, the highest term of $f$ has the same parity as $f$ itself. Thus we can reformulate the previous result in the following way.

\begin{lemma}
\label{lem:highest_term_represents_parity_when_easy}
Let $\ell$ be a lamination represented by a single nonperipheral curve of weight~1. Let $(Q,{\bf q})$ be the unique element of $\mathbb{Z}/4\mathbb{Z} \times (\mathbb{Z}/2\mathbb{Z})^n$ such that $[\mathbb{I}^\omega(\ell)]^H \in (\mathbf{Z}^\omega\{\mathscr{Z}\})_{(Q,{\bf q})}$. Then $\mathbb{I}^\omega(\ell) \in (\mathbf{Z}^\omega\{\mathscr{Z}\})_{(Q,{\bf q})}$. \qed
\end{lemma}

We now proceed to consider more general laminations. We begin with a couple of technical lemmas.

\begin{lemma}
\label{lem:even_power_is_good}
Let $\ell$ be a lamination consisting of a single nonperipheral curve of weight~1. Then $\mathbb{I}^\omega(\ell)^{2b} \in (\mathbf{Z}^\omega\{\mathscr{Z}\})_{(0,{\bf 0})}$, for every nonnegative integer $b$.
\end{lemma}

\begin{proof}
Let $\mu_i = 2a_i$ be twice the coordinate of~$\ell$ associated to the edge~$i$. Then $\mu_i$ is an integer, and Proposition~\ref{prop:highest} says
\[
[\mathbb{I}^\omega(\ell)]^H = [Z_1^{\mu_1} \cdots Z_n^{\mu_n}].
\]
We have
\[
[\mathbb{I}^\omega(\ell)^{2b}]^H = ([\mathbb{I}^\omega(\ell)]^H)^{2b} = [Z_1^{2b\mu_1} \cdots Z_n^{2b\mu_n}].
\]
Note that 
\[
[Z_1^{2b\mu_1} \cdots Z_n^{2b\mu_n}] = \omega^{-\sum_{i<j} \varepsilon_{ij} (2b\mu_i)(2b\mu_j)} Z_1^{2b\mu_1} \cdots Z_n^{2b\mu_n}.
\]
Since $2b\mu_i \equiv 0$ (mod $2$) for each $i=1,\ldots,n$ and thus $-\sum_{i<j} \varepsilon_{ij} (2b\mu_i)(2b\mu_j) \equiv 0$ (mod $4$), one easily sees $[Z_1^{2b\mu_1} \cdots Z_r^{2b\mu_r}] \in (\mathbf{Z}^\omega\{\mathscr{Z}\})_{(0,{\bf 0})}$, from Definition~\ref{def:parity_of_Laurent_omega-monomials}.

On the other hand, Lemma~\ref{lem:highest_term_represents_parity_when_easy} tells us that $\mathbb{I}^\omega(\ell) \in (\mathbf{Z}^\omega\{\mathscr{Z}\})_{(Q,{\bf q})}$ for a unique $(Q,{\bf q})$. Therefore, by applying Lemma~\ref{lem:multiplication_of_parity_subspaces} repeatedly, we see that there is a unique $(Q',{\bf q}')$ such that $\mathbb{I}^\omega(\ell)^{2b} \in (\mathbf{Z}^\omega\{\mathscr{Z}\})_{(Q',{\bf q}')}$. We just showed that $[\mathbb{I}^\omega(\ell)^{2b}]^H$ belongs to $(\mathbf{Z}^\omega\{\mathscr{Z}\})_{(0,{\bf 0})}$, hence $(Q',{\bf q}')=(0,{\bf 0})$.
\end{proof}

\begin{lemma}
\label{lem:odd_power_is_good}
Let $\ell$ be a lamination consisting of a single nonperipheral curve of weight~1. Then there exists $(Q,{\bf q})$ with $\mathbb{I}^\omega(\ell)^{2b+1} \in (\mathbf{Z}^\omega\{\mathscr{Z}\})_{(Q,{\bf q})}$ for all nonnegative integers $b$.
\end{lemma}

\begin{proof}
Lemma~\ref{lem:highest_term_represents_parity_when_easy} tells us that $\mathbb{I}^\omega(\ell) \in (\mathbf{Z}^\omega\{\mathscr{Z}\})_{(Q,{\bf q})}$ for a unique $(Q,{\bf q})$. Suppose $b$ is a positive integer. Then Lemma~\ref{lem:even_power_is_good} says $\mathbb{I}^\omega(\ell)^{2b} \in (\mathbf{Z}^\omega\{\mathscr{Z}\})_{(0,{\bf 0})}$. Finally, Lemma~\ref{lem:multiplication_by_zero_parity} says that the product of $\mathbb{I}^\omega(\ell)$ and $\mathbb{I}^\omega(\ell)^{2b}$ belongs to $(\mathbf{Z}^\omega\{\mathscr{Z}\})_{(Q,{\bf q})}$.
\end{proof}

\begin{lemma}
\label{lem:highest_term_represents_parity_when_multiple_of_easy}
Let $\ell$ be a lamination consisting of a single nonperipheral curve of weight~1, and let $k$ be a positive integer. Let $(Q,{\bf q})$ be the unique element of $\mathbb{Z}/4\mathbb{Z} \times (\mathbb{Z}/2\mathbb{Z})^n$ such that $[\mathbb{I}^\omega(k\ell)]^H \in (\mathbf{Z}^\omega\{\mathscr{Z}\})_{(Q,{\bf q})}$. Then $\mathbb{I}^\omega(k\ell) \in (\mathbf{Z}^\omega\{\mathscr{Z}\})_{(Q,{\bf q})}$.
\end{lemma}

\begin{proof}
Recall that for any positive integer $k$, we have
\[
\mathbb{I}^\omega(k\ell) = F_k(\mathbb{I}^\omega(\ell))
\]
where $F_k(t) \in \mathbb{Z}[t]$ is the $k$th Chebyshev polynomial. Notice that $F_k$ has only terms of degrees that are of same parity (in the usual sense) as $k$. That is, for odd~$k$, \ $F_k$ has only odd degree terms, and for even~$k$, \ $F_k$ has only even degree terms. For even $k$, Lemma~\ref{lem:even_power_is_good} implies $\mathbb{I}^\omega(k\ell) \in (\mathbf{Z}^\omega\{\mathscr{Z}\})_{(0,{\bf 0})}$. For odd $k$, Lemma~\ref{lem:odd_power_is_good} implies $\mathbb{I}^\omega(k\ell) \in (\mathbf{Z}^\omega\{\mathscr{Z}\})_{(Q,{\bf q})}$. It follows that for any~$k$ there exists $(Q,{\bf q})$ such that $\mathbb{I}^\omega(k\ell) \in ( \mathbf{Z}^\omega \{ \mathscr{Z} \} )_{ (Q,{\bf q}) }$. It is easy to see that this statement implies the lemma.
\end{proof}

\begin{lemma}
\label{lem:highest_term_represents_parity_when_single_peripheral}
Let $p$ be a lamination consisting of a single peripheral curve with arbitrary weight. Let $(Q,{\bf q})$ be the unique element of $\mathbb{Z}/4\mathbb{Z} \times (\mathbb{Z}/2\mathbb{Z})^n$ such that $[\mathbb{I}^\omega(p)]^H \in (\mathbf{Z}^\omega\{\mathscr{Z}\})_{(Q,{\bf q})}$. Then $\mathbb{I}^\omega(p) \in (\mathbf{Z}^\omega\{\mathscr{Z}\})_{(Q,{\bf q})}$.
\end{lemma}

\begin{proof}
This follows trivially from the fact that, in this case, $\mathbb{I}^\omega(\ell)$ equals its highest term.
\end{proof}

\subsection{Proof of the main result}

In this section, we prove the results stated in the introduction.

\begin{theorem}
\label{thm:main}
There exists a map $\widehat{\mathbb{I}}^q=\widehat{\mathbb{I}}_T^q:\mathcal{A}_{SL_2,S}(\mathbb{Z}^t)\rightarrow\mathcal{X}_T^q$ satisfying the following properties:
\begin{enumerate}
\item The map is identified with Fock and Goncharov's duality map in the classical limit: $\widehat{\mathbb{I}}^{\, 1}(\ell)=\mathbb{I}(\ell)$.

\item The highest term of $\widehat{\mathbb{I}}^q(\ell)$ is the Weyl ordering
\[
q^{-\sum_{i<j}\varepsilon_{ij}a_ia_j}X_1^{a_1}\dots X_n^{a_n}
\]
where $a_i$ is the coordinate of~$\ell$ associated to the edge~$i$.

\item Each $\widehat{\mathbb{I}}^q(\ell)$ is a Laurent polynomial in the variables $X_i$ with coefficients in $\mathbb{Z}[q,q^{-1}]$.

\item Let $*$ be the canonical involutive antiautomorphism of $\mathcal{X}_T^q$ that fixes each $X_i$ and sends~$q$ to~$q^{-1}$. Then $*\widehat{\mathbb{I}}^q(\ell)=\widehat{\mathbb{I}}^q(\ell)$.

\item For any $\ell$,~$\ell'\in\mathcal{A}_{SL_2,S}(\mathbb{Z}^t)$, we have 
\[
\widehat{\mathbb{I}}^q(\ell)\widehat{\mathbb{I}}^q(\ell')=\sum_{\ell''\in\mathcal{A}_{SL_2,S}(\mathbb{Z}^t)}c^q(\ell,\ell';\ell'')\widehat{\mathbb{I}}^q(\ell'')
\]
where $c^q(\ell,\ell';\ell'')\in\mathbb{Z}[q,q^{-1}]$ and only finitely many terms are nonzero.

\item Let $q$ be a primitive $N$th root of unity for odd $N$. Then we have the identity 
\[
\widehat{\mathbb{I}}^q(N\cdot\ell)(X_1,\dots,X_n)=\widehat{\mathbb{I}}^{\, 1}(\ell)(X_1^N,\dots,X_n^N).
\]

\item Let $\ell$ be a point of the tropical space $\mathcal{A}_{SL_2,S}(\mathbb{Z}^t)$ with coordinates $b_1,\dots,b_n$, and let $a_1,\dots,a_n$ be integers satisfying $\sum_j\varepsilon_{ij}a_j=0$ for all $i\in I$. If $\ell'\in\mathcal{A}_{SL_2,S}(\mathbb{Z}^t)$ has coordinates $a_i+b_i$ for~$i\in I$, then 
\[
\widehat{\mathbb{I}}^q(\ell') = q^{-\sum_{i<j}\varepsilon_{ij}a_ia_j}\prod_iX_i^{a_i}\cdot\widehat{\mathbb{I}}^q(\ell).
\]
\end{enumerate}
For triangulations $T$ and $T'$, the maps $\widehat{\mathbb{I}}_T^q$ and $\widehat{\mathbb{I}}_{T'}^q$ are related by $\widehat{\mathbb{I}}_T^q=\Phi_{TT'}^q\circ\widehat{\mathbb{I}}_{T'}^q$.
\end{theorem}

\begin{proof}
For $\ell\in\mathcal{A}_{SL_2,S}(\mathbb{Z}^t)$, we define $\widehat{\mathbb{I}}^q(\ell)=\mathbb{I}^\omega(\ell)$ where $\omega$ is a fourth root of $q$. Let us check that this gives a map $\widehat{\mathbb{I}}^q:\mathcal{A}_{SL_2,S}(\mathbb{Z}^t)\rightarrow\mathcal{X}_T^q$ with the required properties.

1. If $q=1$ then $\omega$ is a fourth root of unity. We will show in part~3 that the coefficients of $\widehat{\mathbb{I}}^q(\ell)$ are Laurent polynomials in $\omega^4$. Therefore $\widehat{\mathbb{I}}^q(\ell)$ is unchanged if we assume $\omega=1$. We saw in~Proposition~\ref{prop:classicallimit} that $\mathbb{I}^\omega$ can be identified with the map $\mathbb{I}$ defined by Fock and Goncharov when $\omega=1$.

2. We saw in~Proposition~\ref{prop:highest} that the highest term of $\mathbb{I}^\omega(\ell)$ is $[Z_1^{\mu_1}\dots Z_n^{\mu_n}]$ where $\mu_i=2a_i$ is twice the coordinate of~$\ell$ associated to the edge $i$. This equals 
\begin{align*}
[Z_1^{\mu_1}\dots Z_n^{\mu_n}] &= [Z_1^{2a_1}\dots Z_n^{2a_n}] \\
&= \omega^{-\sum_{i<j}\varepsilon_{ij}(2a_i)(2a_j)}Z_1^{2a_1}\dots Z_n^{2a_n} \\
&= \omega^{-4\sum_{i<j}\varepsilon_{ij}a_ia_j}Z_1^{2a_1}\dots Z_n^{2a_n} \\
&= q^{-\sum_{i<j}\varepsilon_{ij}a_ia_j}X_1^{a_1}\dots X_n^{a_n}
\end{align*}
as desired.

3. Let $\ell\in\mathcal{A}_{SL_2,S}(\mathbb{Z}^t)$. Then we can write
\[
\ell = k_1 \ell_1 + \cdots + k_r \ell_r + p_1 + \cdots + p_s
\]
where $\ell_1,\ldots,\ell_r$ are represented by nonperipheral curves of weight~1, $p_1,\ldots,p_s$ by peripheral curves, $k_1,\ldots,k_r$ are positive integers, and the laminations appearing in the decomposition are mutually disjoint and mutually non-homotopic. We can find unique elements $(Q_1,{\bf q}_1), \ldots, (Q_r, {\bf q}_r)$, $(Q_1', {\bf q}_1'), \ldots, (Q_s', {\bf q}_s')$ of $\mathbb{Z}/4\mathbb{Z} \times (\mathbb{Z}/2\mathbb{Z})^n$ such that
\[
[\mathbb{I}^\omega(k_i\ell_i)]^H \in (\mathbf{Z}^\omega\{\mathscr{Z}\})_{(Q_i,{\bf q}_i)},
\quad
[\mathbb{I}^\omega(p_j)]^H \in (\mathbf{Z}^\omega\{\mathscr{Z}\})_{(Q_j',{\bf q}_j')}
\]
for all $i$ and~$j$. Then Lemmas \ref{lem:highest_term_represents_parity_when_multiple_of_easy} and \ref{lem:highest_term_represents_parity_when_single_peripheral} tell us that
\[
\mathbb{I}^\omega(k_i\ell_i) \in (\mathbf{Z}^\omega\{\mathscr{Z}\})_{(Q_i,{\bf q}_i)},
\quad
\mathbb{I}^\omega(p_j) \in (\mathbf{Z}^\omega\{\mathscr{Z}\})_{(Q_j',{\bf q}_j')}.
\]
By repeated use of Lemma~\ref{lem:multiplication_of_parity_subspaces}, we see that there exists a unique $(Q,{\bf q}) \in \mathbb{Z}/4\mathbb{Z} \times (\mathbb{Z}/2\mathbb{Z})^n$ such that
\[
\mathbb{I}^\omega(\ell) = \mathbb{I}^\omega(k_1 \ell_1) \dots \mathbb{I}^\omega(k_r \ell_r) \mathbb{I}^\omega(p_1) \dots \mathbb{I}^\omega(p_s)
\in (\mathbf{Z}^\omega\{\mathscr{Z}\})_{(Q,{\bf q})}.
\]
Moreover, we can say what $(Q,{\bf q})$ is. This last equation implies in particular that 
\[
[\mathbb{I}^\omega(\ell)]^H \in (\mathbf{Z}^\omega\{\mathscr{Z}\})_{(Q,{\bf q})}.
\]
From part~2, we know that 
\[
[\mathbb{I}^\omega(\ell)]^H \in (\mathbf{Z}^\omega\{\mathscr{Z}\})_{(0,{\bf 0})}.
\]
So $(Q,{\bf q}) = (0,{\bf 0})$. This means that $\mathbb{I}^\omega(\ell)$ is a Laurent polynomial in $X_1=Z_1^2,\dots,X_n=Z_n^2$ whose coefficients are Laurent polynomials in $q=\omega^4$ with integer coefficients.

4. The canonical antiautomorphism $*$ on $\mathcal{Z}_T^\omega$ restricts to the canonical antiautomorphism on the subalgebra $\mathcal{X}_T^q\subseteq\mathcal{Z}_T^\omega$. By Proposition~\ref{prop:starinvariance}, $\mathbb{I}^\omega(\ell)$ is invariant under this map for any lamination $\ell\in\mathcal{A}_{SL_2,S}(\mathbb{Z}^t)\subseteq\mathcal{A}_L(S,\mathbb{Z})$.

5. This product formula was proved in~Proposition~\ref{prop:productexpansion}. We need to show that the coefficients in this expansion are Laurent polynomials in $q$ and that the laminations $\ell''$ appearing in the decomposition lie in $\mathcal{A}_{SL_2,S}(\mathbb{Z}^t)$. Note that if $\ell$,~$\ell'\in\mathcal{A}_{SL_2,S}(\mathbb{Z}^t)$, then the expressions $\mathbb{I}^\omega(\ell)$ and $\mathbb{I}^\omega(\ell')$ both have parity $(0,\mathbf{0})$, and therefore so does their product. Let us expand this product as 
\[
\sum_{j=1}^k c^\omega(\ell,\ell';\ell_j'')\mathbb{I}^\omega(\ell_j'')
\]
for distinct laminations $\ell_1'',\dots,\ell_k''\in\mathcal{A}_L(S,\mathbb{Z})$ and $0\neq c^\omega(\ell,\ell';\ell_j'')\in\mathbb{Z}[\omega^2,\omega^{-2}]$. We can impose a lexicographic total ordering $\geq$ on the set of all commutative monomials $Z_1^{\mu_1}\dots Z_n^{\mu_n}$ so that $[\mathbb{I}^1(\ell_j'')]^H$ is the highest term of~$\mathbb{I}^1(\ell_j'')$ with respect to this total ordering. By Lemma~\ref{lem:highestdetermines}, the highest terms of the $\mathbb{I}^1(\ell_j'')$ are distinct, so we may assume 
\[
[\mathbb{I}^1(\ell_1'')]^H>[\mathbb{I}^1(\ell_2'')]^H>\dots >[\mathbb{I}^1(\ell_k'')]^H.
\]
Consider the expression $c^\omega(\ell,\ell';\ell_1'') [\mathbb{I}^\omega(\ell_1'')]^H$. It cannot cancel with any other term in the sum, so it must have parity $(0,\mathbf{0})$. A term of $c^\omega(\ell,\ell';\ell_1'')$ has parity $(Q,{\bf 0})$ for some $Q\in \mathbb{Z}/4\mathbb{Z}$, so Lemma~\ref{lem:multiplication_by_zero_parity} implies that $[\mathbb{I}^\omega(\ell_1'')]^H$ must have parity $(-Q, {\bf 0})$. This means that all exponents in the monomial $[\mathbb{I}^\omega(\ell_1'')]^H$ are even, and hence $\ell_1''\in\mathcal{A}_{SL_2,S}(\mathbb{Z}^t)$. By part~3 we see that $\mathbb{I}^\omega(\ell_1'')$ has parity $(0,{\bf 0})$, hence so does $[\mathbb{I}^\omega(\ell_1'')]^H$. Hence $c^\omega(\ell,\ell';\ell''_1)$ is of parity $(0,{\bf 0})$. Equivalently, it is a Laurent polynomial in~$q$ with integral coefficients. Applying Lemma~\ref{lem:multiplication_by_zero_parity} again, we see that $c^\omega(\ell,\ell';\ell_j'')\mathbb{I}^\omega(\ell_1'')$ has parity~$(0,\mathbf{0})$, and therefore so does the sum 
\[
\sum_{j=2}^k c^\omega(\ell,\ell';\ell_j'')\mathbb{I}^\omega(\ell_j'').
\]
Arguing as before, we see that $\ell_2''\in\mathcal{A}_{SL_2,S}(\mathbb{Z}^t)$ and $c^\omega(\ell,\ell';\ell''_2)$ is a Laurent polynomial in~$q$. Continuing in this way, we see that all $\ell_j''$ lie in the space~$\mathcal{A}_{SL_2,S}(\mathbb{Z}^t)$ and all $c^\omega(\ell,\ell';\ell''_j)$ lie in the ring $\mathbb{Z}[q,q^{-1}]$. Let $c^q(\ell,\ell';\ell''_j) := c^\omega(\ell,\ell';\ell''_j)$ for each $j$ to get the desired result.

6. Suppose $\ell\in\mathcal{A}_L(S,\mathbb{Z})$ is a lamination on~$S$ consisting of a single nonperipheral loop of weight~1. Let $q$ be a primitive $N$th root of unity for odd $N$, and choose $A$ so that $q=A^{-2}$. Then $A^N=\pm1$. If $A^N=-1$, then we can replace $A$ by $\tilde{A}=-A$, and we have $q=\tilde{A}^{-2}$ and $\tilde{A}^N=1$. Thus we can assume that $A^N=1$. It follows that 
\[
\kappa\coloneqq A^{N^2}=1.
\]
Choose $\omega$ so that $A=\omega^{-2}$. Then $\omega^{N^2}=\pm1$. As before, we can replace $\omega$ by $-\omega$ if necessary and assume that $\omega^{N^2}=1$. Then 
\[
\iota\coloneqq\omega^{N^2}=1.
\]
Note that $A^4$ is a primitive $N$th root of unity since $N$ is odd. We can therefore apply Theorem~\ref{thm:BWFrobenius} with our choices of $A$ and $\omega$. Let $K$ be a framed link in $S\times[0,1]$ with constant elevation and the vertical framing that projects to $\ell$. Then Theorem~\ref{thm:BWFrobenius} implies 
\begin{align*}
F_N(\Tr_T^\omega([K]))(Z_1,\dots,Z_n) &= \Tr_T^\omega(F_N([K]))(Z_1,\dots,Z_n) \\
&= \Tr_T^\iota([K])(Z_1^N,\dots,Z_n^N).
\end{align*}
Consider a lamination $k\ell$ consisting of a single nonperipheral curve of weight~$k$. It is well known that Chebyshev polynomials satisfy $F_s\circ F_t=F_{st}$ for any positive integers $s$ and~$t$. Using this fact and the previous calculation, we see that 
\begin{align*}
\mathbb{I}^\omega(Nk\ell)(Z_1,\dots,Z_n) &= F_{kN}(\Tr_T^\omega([K]))(Z_1,\dots,Z_n) \\
&= F_k(F_N(\Tr_T^\omega([K])))(Z_1,\dots,Z_n) \\
&= F_k(\Tr_T^1([K]))(Z_1^N,\dots,Z_n^N) \\
&= \mathbb{I}^1(k\ell)(Z_1^N,\dots,Z_n^N).
\end{align*}
On the other hand, suppose $\ell\in\mathcal{A}_L(S,\mathbb{Z})$ is a lamination on~$S$ consisting of a single peripheral loop. Let $\mu_i=2a_i$ be twice the coordinate of~$\ell$ associated to the edge $i$ of the ideal triangulation. Then 
\begin{align*}
\mathbb{I}^\omega(N\ell)(Z_1,\dots,Z_n) &= \omega^{-N^2\sum_{i<j}\mu_i\mu_j\varepsilon_{ij}} Z_1^{N\mu_1}\dots Z_n^{N\mu_n} \\
&= Z_1^{N\mu_1}\dots Z_n^{N\mu_n} \\
&= \mathbb{I}^1(\ell)(Z_1^N,\dots,Z_n^N).
\end{align*}
Finally, if $\ell$ is any lamination, we can write $\ell=\sum_ik_i\ell_i$ where $\ell_i$ are the curves of~$\ell$ with each homotopy class of curves appearing at most once in the sum and $k_i\in\mathbb{Z}$. Then the above calculations imply 
\begin{align*}
\mathbb{I}^\omega(N\ell)(Z_1,\dots,Z_n) &= \prod_i\mathbb{I}^\omega(Nk_i\ell_i)(Z_1,\dots,Z_n) \\
&= \prod_i\mathbb{I}^1(k_i\ell_i)(Z_1^N,\dots,Z_n^N) \\
&= \mathbb{I}^1(\ell)(Z_1^N,\dots,Z_n^N).
\end{align*}
Restricting to $\ell\in\mathcal{A}_{SL_2,S}(\mathbb{Z}^t)$ yields the desired result.

7. Let $a_1,\dots,a_n$ be integers satisfying $\sum_j\varepsilon_{ij}a_j=0$ for all elements $i\in I$. Consider the lamination $\ell''$ having coordinates $\{a_i\}$. Let $k$ be an edge of the ideal triangulation of our surface that separates two triangles. The union of these triangles is a quadrilateral and the number $\sum_j\varepsilon_{kj}a_j$ equals the signed weight of curves of the lamination $\ell''$ that go diagonally across this quadrilateral, connecting opposite sides. Therefore the condition $\sum_j\varepsilon_{ij}a_j=0$ implies that no curve of $\ell''$ goes diagonally across a quadrilateral. Equivalently, each curve of $\ell''$ is peripheral.

Since each curve of $\ell''$ is peripheral, we can deform the curves of $\ell''$ into a small neighborhood of the punctures so that they do not intersect the curves of $\ell$. Then the union of the curves of $\ell$ and $\ell''$ represents the unique lamination with coordinates $a_i+b_i$ for $i\in I$. Call this lamination~$\ell'$. We know that $\widehat{\mathbb{I}}^q(\ell'')$ equals its highest term:
\[
\widehat{\mathbb{I}}^q(\ell'')=q^{-\sum_{i<j}\varepsilon_{ij}a_ia_j}\prod_iX_i^{a_i}.
\]
Therefore 
\begin{align*}
\widehat{\mathbb{I}}^q(\ell') &= \widehat{\mathbb{I}}^q(\ell'')\widehat{\mathbb{I}}^q(\ell) \\
&= q^{-\sum_{i<j}\varepsilon_{ij}a_ia_j}\prod_iX_i^{a_i}\cdot\widehat{\mathbb{I}}^q(\ell)
\end{align*}
as desired.
\end{proof}

\section*{Acknowledgments}
\addcontentsline{toc}{section}{Acknowledgements}

The first author thanks the Korea Institute for Advanced Study for hosting him in Seoul where much of this work was completed.

\end{document}